\documentclass[11pt,letterpaper]{article}
\usepackage[english]{babel}
\usepackage{amsmath,amsthm}
\usepackage{amsfonts, mathrsfs}
\usepackage{amssymb}
\usepackage{verbatim}
\usepackage{graphicx}
\usepackage{color}
\usepackage{tikz-cd}
\usepackage{enumitem}
\usepackage{microtype}
\usepackage{url}
\usepackage{hyperref}
\hypersetup{colorlinks = true, linkcolor = blue, citecolor = blue, urlcolor = blue}
\usepackage[normalem]{ulem}
\usepackage{float}
\usepackage{graphicx}
\graphicspath{{figures/}}

\setitemize{itemsep=1ex,leftmargin=*}
\allowdisplaybreaks

\newtheorem{theorem}{Theorem}

\newtheorem{prop}[theorem]{Proposition}
\newtheorem{lemma}[theorem]{Lemma}

\theoremstyle{definition}
\newtheorem{defn}{Definition}
\newtheorem{remark}[defn]{Remark}
\newtheorem{remarks}[defn]{Remarks}

\newcommand{\var}   {{\rm Var} \mspace{1mu}}
\newcommand{\cov}   {{\rm Cov} \mspace{1mu}}
\newcommand{\eqlaw} {\stackrel{\mathcal{D}}{=}}

\begin{document}

\title{Opinion dynamics on directed complex networks}

\author{Nicolas Fraiman \and Tzu-Chi Lin \and Mariana Olvera-Cravioto}

\maketitle
\begin{abstract}
We propose and analyze a mathematical model for the evolution of opinions on directed complex networks. Our model generalizes the popular DeGroot and Friedkin-Johnsen models by allowing vertices to have attributes that may influence the opinion dynamics. We start by establishing sufficient conditions for the existence of a stationary opinion distribution on any fixed graph, and then provide an increasingly detailed characterization of its behavior by considering a sequence of directed random graphs having a local weak limit. Our most explicit results are obtained for graph sequences whose local weak limit is a marked Galton-Watson tree, in which case our model can be used to explain a variety of phenomena, e.g., conditions under which consensus can be achieved, mechanisms in which opinions can become polarized, and the effect of disruptive stubborn agents on the formation of opinions. 
\vspace{5mm}


\end{abstract}

\section{Introduction}\label{sec:intro}

Today's era of political polarization has fueled great interest in the development of models for the evolution of opinions on social networks. A topic that has been of interest to social scientists and mathematicians for a long time, now requires a new generation of models that can move away from the traditional notion of consensus to the more complex phenomenon of polarization. The fast growth that the field of complex networks has experienced in the last couple of decades can now be used to provide tractable models that incorporate both the random nature of human interactions and the complex structure of social networks, yielding models that have high explanatory power in terms of commonly accepted features such as personal relationships, individual traits, and the influence of the media. 

This paper presents a model for the evolution of opinions on a social network where individuals are periodically influenced by their acquaintances in the network, by the signals they receive from the media, and by their own personal beliefs. It is worth pointing out that we think of a social network broadly, in a way that includes both person to person relations as well as virtual connections (e.g., Facebook, X). The goal of our model is to explain how opinions evolve over time and reach a stable distribution, which may look like consensus in some cases, but that may also exhibit a high degree of polarization in others. 

Our approach starts with the definition of a stochastic process on a finite directed graph, where we can give simple and natural conditions for the existence of a stationary distribution for the opinions held by individuals in the network. On small graphs, this stationary distribution would naturally depend to great extent on the specific graph on which the opinions are evolving. However, today's social networks are extremely large, and random phenomena on large complex networks can be analyzed using techniques from random graph theory. In the context of our opinion model, this means that we can study in great detail the stationary distribution of the opinion of a typical individual, or equivalently, the proportion of individuals having a given opinion, in a large class of random graph models. Finally, our last set of results prove that for a specific class of random graph models, the stationary opinion can be shown to achieve high levels of consensus under some conditions, and high levels of polarization under others. Moreover, these conditions depend on the types of signals the media broadcasts, the strength of the influence of close neighbors, and the propensity of individuals to choose to listen only to opinions that agree with their own beliefs, i.e., the cognitive phenomenon known as {\em confirmation bias} \cite{nickerson1998confirmation}. A lengthy discussion of related opinion dynamics models is given in Section~\ref{sec:literature}.

The remainder of the paper is organized as follows:
Section~\ref{sec:model} includes the description of the opinion model on a fixed graph, and Section~\ref{sec:stationary_bahavior} discusses its stationary distribution.  Section~\ref{sec:typical_bahavior} introduces the notion of a typical vertex in a sequence of locally convergent random graphs. Section~\ref{sec:explicit_characterizations} gives an explicit characterization of the typical opinion when the graph's local weak limit is a tree, and then specializes to a marked Galton-Watson tree. Subsections~\ref{sec:consensus} through~\ref{sec:memory} discuss different types of behaviors that the model can capture under various parameter settings. We also include in Section~\ref{sec:literature} a literature review, and Section~\ref{sec:proofs} contains most of the proofs. Finally, the appendix includes additional proofs for the mean and variance formulas in the most general case.
\section{Model Description}\label{sec:model}

Our model for a social network consists of a directed graph where individuals are represented by its vertices and the edges correspond to relationships, e.g., family relatives, friends, co-workers, etc. Since we believe that trust is not always symmetric between any two individuals, these acquaintances are assumed to be directed, with an edge $(i,j)$ meaning that individual $j$ trusts, or listens, to the opinion of individual $i$. In graph terminology, the set of inbound neighbors of vertex $j$ corresponds to the group of people that he/she is willing to listen to, and the in-degree $d_j^-$ of vertex $j$ corresponds to the cardinality of this set.  Symmetrically, vertex $i$ has a set of outbound neighbors who will listen to his/her opinion, and the cardinality of this set is its out-degree $d_i^+$. For a fixed graph, we will use $G(V,E)$ to denote the directed graph having vertex set $V$ and set of directed edges $E$. 

Opinions on the social network are assumed to evolve over time at discrete time-steps. We focus on one topic at a time, which can take values on the continuous interval $[-1, 1]$ (the multi topic model will be considered in a future work). The opinion of individual $i$ at time $k$ will be denoted as $R_i^{(k)}$. Each individual is assumed to have a set of attributes, including: a ``natural" predisposition on the topic being discussed, which we will call the internal opinion, and a set of weights that determine the amount of trust that they give to each of their inbound neighbors. In terms of the graph, these individual characteristics become vertex attributes that remain constant throughout time. In addition, we assume the entire social network is exposed to media signals that may influence their opinions on the topic.  However, what type of media signals each individual chooses to listen to may depend on their vertex attributes, hence allowing us to model a special kind of confirmation bias known as selective exposure. 

We assume that at each time step, each individual in the network listens to the opinions of all its inbound neighbors and weighs them according to the trust he/she has given to  these neighbors. Simultaneously, each individual is exposed to the external signals (e.g. media signal) for that time step, and then proceeds to update their own opinion by taking into account:
\begin{enumerate} \itemsep 1pt \topsep 0pt
\item the weighted average of their inbound neighbors' opinions,
\item the external signal received during that period, and
\item their own opinion from the previous time step.
\end{enumerate}
The resulting recursion can be written as:
\begin{equation} \label{eq:OpinionRec}
R_i^{(k+1)} = \sum_{r=1}^{d_i^-} c(i,r) R_{\ell(i,r)}^{(k)} + W_i^{(k)} + (1-c -d) R_i^{(k)}, \qquad i \in V, \qquad k \geq 0,
\end{equation}
where $\ell(i,r)$ is the vertex label of the $r$th inbound neighbor of individual $i$, $c(i,r)$ is the weight/trust that individual $i$ places on its $r$th inbound neighbor, and $W_i^{(k)}$ is the external signal received by individual $i$ during the $k$th time period; the recursion is initialized with $\{ R_i^{(0)}: i \in V\}$. The neighbor weights are assumed to satisfy:
$$
\sum_{r=1}^{d_i^-} c(i,r) = c \qquad \text{for all } i \in V \text{ with } d_i^- > 0,
$$
and $c \in [0,1)$, $d \in (0,1]$ are parameters satisfying $c+d \leq 1$. Before describing the nature of the external signals $\{ W_i^{(k)}\}$, recall that individuals are allowed to have a set of attributes that may influence what type of external signal they receive. We will use $q_i \in [-1,1]$ to denote the natural propensity or internal opinion of individual $i$. If we consider that individuals with no inbound neighbors are not susceptible to being influenced, we can think of labeling such individuals as {\em stubborn agents}, and include a label of the form $s_i \in \{0,1\}$ as part of their attributes as well. Bots, or artificial individuals whose opinions cannot be influenced but who can influence others can also be modeled using the labels $s_i$. In the sequel we will also consider scenarios where the graph representing the social network is random, in which case we may want to include attributes that can influence an individual's number of inbound/outbound neighbors. For now, we model attributes not directly related to recursion \eqref{eq:OpinionRec} using a generic vector $\mathbf{a}_i$ taking values in some Polish space $\mathcal{S}'$. The full vector of attributes of individual $i$ takes the form:
$$
\mathbf{x}_i = (\mathbf{a}_i, q_i, s_i, d_i^-, c(i,1) 1(d_i^- \geq 1), c(i,2) 1(d_i^- \geq 2), \dots) \in \mathcal{S},
$$
where $\mathcal{S} := \mathcal{S}' \times [-1,1] \times \{0,1\} \times \mathbb{N} \times [0,1]^\infty$.

Now that we have defined the vertex attributes, the external signals $\{ W_i^{(k)}: i \in V, k \geq 1\}$ are assumed to be independent of each other, with the per-individual external signals $\{ W_i^{(k)}: k \geq 1\}$ conditionally i.i.d.~given $\mathbf{x}_i$ for each $i \in V$, and satisfying:
$$
\left| W_i^{(k)} \right| \leq d + c - \sum_{r=1}^{d_i^-} c(i,r).
$$

In the explicit computations in Section~\ref{sec:consensus} onwards, we will use the following form for the external signals:
\begin{equation} \label{eq:MediaSignals}
W_i^{(k)} = q_i \left( c-\sum_{r=1}^{d_i^-} c(i,r) \right) + d Z_i^{(k)},
\end{equation} 
with the random variables $\{Z_i^{(k)}: k \geq 1, i \in V\}$ independent of each other and taking values in $[-1,1]$. We will refer to the $\{Z_i^{(k)}: k \geq 1, i \in V\}$ as the media signals, and point out that the term multiplying $q_i$ is non-zero only for vertices with no inbound neighbors (it replaces the neighbors' influence with the vertex's internal opinion).  

The generality of the external signals as described above allows us great modeling flexibility, as some of the examples below illustrate:
\begin{itemize} \itemsep 1pt \topsep 0pt
\item Choosing the $\{Z_i^{(k)}: k \geq 1, i \in V\}$ to be i.i.d.~random variables independent of everything else, can be interpreted as ``pure noise", or a media signal that is received by all the non-stubborn agents in the network in a similar way, regardless of their natural predispositions. 

\item When $d = 0$, a case we exclude in this work, the model is known as the DeGroot model \cite{degroot1974reaching}. The model with $c+d = 1$, $Z_i^{(k)} = q_i$ for all $k \geq 0$ and all $i \in V$ is known as the Friedkin-Johnsen model \cite{friedkin1990social}. The case where the $\{Z_i^{(k)}: k \geq 1, i \in V\}$ are i.i.d.~and independent of the vertex marks has been studied in \cite{xiao2007distributed,yang2017innovation}.

\item Allowing the distribution of the $\{Z_i^{(k)}: k \geq 1\}$ in \eqref{eq:MediaSignals} to depend on $\mathbf{x}_i$ for each $i \in V$ can result in media signals that are highly correlated with the individual's internal opinion $q_i$, which we use to model selective exposure. 
\item More generally, allowing the  distribution of the $\{Z_i^{(k)}: k \geq 1\}$ in \eqref{eq:MediaSignals} to depend on other parts of the attribute vector $\mathbf{x}_i$, besides the internal opinion $q_i$, can further exploit the individual's characteristics to model a more targeted media signal. 
\end{itemize}

Having now introduced our proposed model for the evolution of opinions on a fixed directed graph $G$, we go on to present our first result regarding the existence of a stationary distribution for the opinions of individuals in the social network. 

\section{Stationary Behavior} \label{sec:stationary_bahavior}

Our opinion model as defined above defines a Markov chain taking values on $[-1,1]^{n}$, where $n = |V|$ is the number of individuals in the social network being modeled, or equivalently, the number of vertices in the graph $G$. Specifically, if we let $\mathbf{R}^{(k)} = (R_1^{(k)}, \dots, R_n^{(k)})$, then the process:
$$
\{ \mathbf{R}^{(k)}: k \geq 0\},
$$
initialized with the vector $\mathbf{R}^{(0)}$, is a Markov chain with time-homogeneous noise vectors $\{ \mathbf{W}^{(k)}: k \geq 1\}$, where $\mathbf{W}^{(k)} = (W_1^{(k)}, \dots, W_n^{(k)})$. 

Hence, the first question one may ask is whether $\{ \mathbf{R}^{(k)}: k \geq 0\}$ has a stationary distribution as $k \to \infty$. The key parameter that ensures the existence of a stationary distribution turns out to be $d > 0$. In other words, as long as the external signals play a role in the formation of opinions, the opinions of all individuals in the social network will attain an equilibrium distribution, which in general will not be concentrated around a single point or vector. 

The existence of a stationary distribution on a fixed graph can be derived from Theorem~1.1 in \cite{Diac_Freed_99}, however, for completeness, we give a short proof in Section~\ref{sec:proofs}. The precise result is given below; $\Rightarrow$ denotes weak convergence on $\mathbb{R}^{|V|}$. 

\begin{theorem}\label{thm:graphlim}
Let $\mathbf{R}^{(k)} = \{R_i^{(k)}: i\in V\}$ satisfy recursion~\eqref{eq:OpinionRec} on a fixed locally finite directed marked graph $G(V,E;\mathscr{A})$ (i.e., $d_i^- + d_i^+ < \infty$ for all $i \in V$) with given vertex attributes $\{\mathbf{x}_i: i \in V\}$. If $d>0$, then there exists a random vector $\mathbf{R} \in [-1,1]^{|V|}$ such that it does not depend on the distribution of $\mathbf{R}^{(0)}$ and
$$
\mathbf{R}^{(k)} \Rightarrow \mathbf{R}, \qquad k \to \infty.
$$
Furthermore, the convergence occurs at a geometric rate. 
\end{theorem}

\begin{remarks} \
\begin{enumerate}[leftmargin=*]
\item As mentioned earlier, the key feature of the model ensuring the existence of a stationary distribution is $d > 0$, since it implies that recursion \eqref{eq:OpinionRec} can be seen as a contraction operator on the underlying graph $G$ under a suitable Wasserstein metric.

\item The graph in Theorem~\ref{thm:graphlim} is only assumed to be locally finite, so it may in fact have an infinite number of vertices. In addition, it is allowed to be disconnected, which is important in our context given the possible presence of stubborn agents. 

\item If we allow $d = 0$, then our model reduces to the well-known DeGroot model \cite{degroot1974reaching}, and the process $\{ \mathbf{R}^{(k)}: k \geq 0\}$ can be thought of as the power iterations of a recursion of the form:
$$
\mathbf{R}^{(k+1)} = \mathbf{R}^{(k)} H = \mathbf{R}^{(0)} H^{k+1},
$$
for some deterministic matrix $H \in [0,1]^{|V|\times|V|}$. Since in this setting the matrix $H$ is usually assumed to be irreducible and stochastic, we have that $\mathbf{R}^{(k)} \to \mathbf{R} = \mathbf{R}^{(0)} \mathbf{v} \mathbf{1}'$ as $k \to \infty$, where $\mathbf{v}$ is the left Perron-Frobenius eigenvector of $H$ and $\mathbf{1}$ is the row vector of ones. Clearly, $\mathbf{R}$ is a deterministic vector that depends on the initial vector of opinions $\mathbf{R}^{(0)}$, and can be interpreted as the limiting vector of (non-random) opinions that individuals in the network will converge to after a large number of interactions. 

\item Unlike the case $d = 0$, the limiting vector $\mathbf{R}$ will be in general random (unless $W_i^{(k)} \equiv w_i$ for all $k \geq 1$ and $i \in V$), and its law corresponds to the stationary distribution of the Markov chain $\{ \mathbf{R}^{(k)}: k \geq 0\}$. This distribution does not depend on the initial vector of opinions $\mathbf{R}^{(0)}$.

\item In general, the components of the vector $\mathbf{R}$ in Theorem~\ref{thm:graphlim} for any fixed graph $G$ will not be either independent, nor identically distributed. The opinions of close neighbors will also tend to be highly correlated, and the topology of the graph $G$ will strongly influence the shape of the distribution of $\mathbf{R}$, especially for small graphs.

\end{enumerate}
\end{remarks}

In order for us to better understand the stationary distribution of $\{\mathbf{R}^{(k)}: k \geq 0\}$, it is useful to assume that the underlying graph is large, since some of the questions we aim to answer (e.g., the existence of consensus or polarization), do not in fact depend on the specific topology of the underlying graph $G$. To continue our analysis, we then replace the fixed graph $G$ with a family of vertex-weighted directed random graphs $\{ G(V_n, E_n; \mathscr{A}_n): n \geq 1\}$. From this point of view, the fixed graph $G$ can be thought of as simply one realization, and any insights about the opinion model that we can obtain for this family of random graphs can be expected to hold with high probability for the fixed graph $G$.

\section{Typical Behavior}\label{sec:typical_bahavior}

As mentioned in the previous section, in order for us to obtain more detailed information about the stationary distribution of the Markov chain $\{ \mathbf{R}^{(k)}: k \geq 0\}$, we will assume from now on that the underlying graph is a realization from a directed random graph model, which we will analyze under the large graph limit as the number of vertices grows to infinity. Under this asymptotic regime, any result we can establish for our opinion model can be interpreted as having a high probability of being true for any sufficiently large graph sampled from the underlying random graph model. 

Although our finer results in the sequel do require that we assume a particular class of random graph models, the results in this section only need the existence of a local weak limit. To start, consider a sequence of vertex-weighted directed random graphs (multigraphs) $\{ G(V_n, E_n; \mathscr{A}_n): n \geq 1\}$, where the $n$th graph in the sequence, $G(V_n, E_n; \mathscr{A}_n)$, has vertex set $V_n = \{ 1, 2, \dots, n\}$ and where each vertex $i \in V_n$ is assumed to have an attribute vector 
$$
\mathbf{Y}_i = (\mathbf{A}_i, Q_i, S_i) \in \mathcal{S}' \times [-1,1] \times \{0,1\}.
$$
The set of attributes is denoted $\mathscr{A}_n = \{ \mathbf{Y}_i: i \in V_n\}$. Since the graph (multigraph) $G(V_n, E_n; \mathscr{A}_n)$ is assumed to be random, the set of edges $E_n$ is not known until the graph is realized, and neither are the in-degrees and out-degrees of the vertices, denoted $D_i^-$ and $D_i^+$, respectively. Since the weights that each individual gives to its inbound neighbors also depends on the in-degree $D_i^-$, we assume that those are sampled after the edges in the graph have been realized. The distribution of the weights is allowed to depend on the vertex attributes $\mathbf{Y}_i$, but the weight vectors
$$
(C(i,1), C(i,2), \dots, C(i,D_i^-)), \quad \sum_{r=1}^{D_i^-} C(i,r) = c, \qquad i \in V_n,
$$
are assumed to be conditionally independent given $\mathscr{A}_n$. Once the set of edges has been sampled, the full mark of vertex $i \in V_n$ takes the familiar form:
$$
\mathbf{X}_i = (\mathbf{A}_i, Q_i, S_i, D_i^-, C(i,1) 1(D_i^- \geq 1), C(i,2) 1(D_i^- \geq 2), \dots) \in \mathcal{S}.
$$
The use of upper case letters is there to emphasize the random nature of these vertex attributes. 

Once we have sampled the random graph $G(V_n, E_n; \mathscr{A}_n)$, we run the Markov chain $\{ \mathbf{R}^{(k)}: k \geq 0\}$ on this fixed graph, starting from an initial vector $\mathbf{R}^{(0)}$ that is assumed to be conditionally independent of the signals given $G(V_n, E_n; \mathscr{A}_n)$. If we use $\mathbf{R} = (R_1, \dots, R_n)$ to denote a random vector distributed according to the stationary distribution of $\{\mathbf{R}^{(k)}: k \geq 0\}$, the object of interest is $R_{I_n}$, where vertex $I_n$ is uniformly chosen in $V_n$. In words, $R_{I_n}$ represents the stationary opinion of a ``typical" individual, but it also represents the average opinion, since conditionally on the graph $G(V_n, E_n; \mathscr{A}_n)$, we have
\begin{equation} \label{eq:LWC}
P\big( R_{I_n} \in A \mid G(V_n, E_n; \mathscr{A}_n) \big) = \frac{1}{n} \sum_{i=1}^n 1( R_i \in A)
\end{equation}
for any measurable set $A \subseteq [-1,1]$. Our use of random graph theory and local weak convergence (see, e.g., \cite{aldous2004objective, Hofstad2}) will allow us to prove that \eqref{eq:LWC} converges as $n \to \infty$ and becomes independent of the particular realization of the graph $G(V_n, E_n; \mathscr{A}_n)$. In other words, the stationary opinion distribution does not depend on the specific social network, provided it is sufficiently large and it is consistent with one of the many random graph models that have local weak limits. 

In order for us to state our main result regarding \eqref{eq:LWC}, we need to define a few concepts from random graph theory. 

\begin{defn}
We say that two simple graphs $G(V, E)$ and $G'(V', E')$ are {\em isomorphic} if there exists a bijection $\theta: V \to V'$ such that edge $(i,j) \in E$ if and only if edge $(\theta(i), \theta(j)) \in E'$.  We say that two multigraphs $G(V,E)$ and $G'(V', E')$ are {\em isomorphic} if there exists a bijection $\theta: V \to V'$ such that $l(i) = l(\theta(i))$ and $e(i,j) = e(\theta(i), \theta(j))$ for all $i \in V$ and all $(i,j) \in E$, where $l(i)$ is the number of self-loops of vertex $i$ and $e(i,j)$ is the number of edges from vertex $i$ to vertex $j$. In both cases, we write $G \simeq G'$.
\end{defn}

\begin{defn}
For any $i\in V_n$, define $G_i^{(k)}(\mathbf{X})$ to be the marked induced subgraph of $G(V_n, E_n; \mathscr{A}_n)$ consisting of vertices $j \in V_n$ such that there exists a directed path from $j$ to $i$ of length at most $k$, along with their marks, in other words, the marked in-component of depth $k$ of vertex $i$. We write $G_i^{(k)}$ for the corresponding graph without the marks. 
\end{defn}

The main assumption that we need in order to prove that \eqref{eq:LWC} converges as $n \to \infty$, is the following notion of a strong coupling \cite{Olvera_21}, which states the existence of a simultaneous construction of the random graph $G(V_n, E_n; \mathscr{A}_n)$ and its local weak limit, which we will denote $\mathcal{G}(\boldsymbol{\mathcal{X}})$, that ensures that the inbound neighborhood of vertex $I_n$ is very similar, including vertex marks, to the inbound neighborhood of a distinguished vertex, called the root $\emptyset$, of a graph $\mathcal{G}_\emptyset(\boldsymbol{\mathcal{X}}) \eqlaw \mathcal{G}(\boldsymbol{\mathcal{X}})$. It is worth mentioning that when we say that $\mathcal{G}(\boldsymbol{\mathcal{X}})$ is the local weak limit of $\{G(V_n, E_n; \mathscr{A}_n): n \geq 1\}$, what we mean is that for each $n \geq 1$ and each vertex $i \in V_n$ we can construct a graph $\mathcal{G}_{\emptyset(i)}(\boldsymbol{\mathcal{X}})$ whose law is equal to that of $\mathcal{G}(\boldsymbol{\mathcal{X}})$ when $i$ is chosen uniformly at random in $V_n$; more precisely,  $\mathcal{G}_{\emptyset(i)}(\boldsymbol{\mathcal{X}}) \eqlaw (\mathcal{G}(\boldsymbol{\mathcal{X}}) | \boldsymbol{\mathcal{Y}}_\emptyset = \mathbf{Y}_i)$, where $\boldsymbol{\mathcal{Y}}_\emptyset$ is the attribute of the root node $\emptyset$. The definition of a strong coupling alludes to the graphs constructed for two vertices uniformly chosen at random in $V_n$.

\begin{defn} \label{def:StrongCoupling}
The sequence $\{ G(V_n, E_n; \mathscr{A}_n): n \geq 1\}$ admits a \emph{strong coupling} with a marked directed graph $\mathcal{G}(\boldsymbol{\mathcal{X}})$ if the following hold:
\begin{itemize}
\item If $I_n$ is a uniformly chosen vertex in $V_n$ then for any fixed $k\geq 1$ and $\epsilon>0$, there exists a coupling between $G_{I_n}^{(k)}(\mathbf{X})$ and  a graph $\mathcal{G}_{\emptyset}(\boldsymbol{\mathcal{X}}) \eqlaw \mathcal{G}(\boldsymbol{\mathcal{X}})$ such that if $\mathcal{G}_{\emptyset}^{(k)}$ denotes the in-component of depth $k$ of the root of $\mathcal{G}_{\emptyset}(\boldsymbol{\mathcal{X}})$, the event
\begin{equation}\label{eq:distk_coupling}
\mathcal{E}_{I_n}^{(k,\epsilon)} = \left\{ \mathcal{G}^{(k)}_{\emptyset}\simeq G_{I_n}^{(k)},\, \theta(\emptyset) = I_n, \,  \bigcap_{i \in \mathcal{G}^{(k)}_{\emptyset}}   \{ \rho(\mathbf{X}_{\theta(i)}, \boldsymbol{\mathcal{X}}_{i} ) \leq \epsilon \} \right\},
\end{equation}
satisfies
$P\left( \mathcal{E}_{I_n}^{(k,\epsilon)} \,\middle|\, \mathscr{A}_n \right) \xrightarrow{P} 1$, as $n \to \infty$, where $\theta$ is the bijection defining $\mathcal{G}_{\emptyset}^{(k)} \simeq G_{I_n}^{(k)}$ and $\rho$ is the metric on the space of full marks $\mathcal{S}$.

\item Given two independent uniformly chosen vertices $I_n, J_n \in V_n$, there exist two independent copies of $\mathcal{G}(\boldsymbol{\mathcal{X}})$, denoted $\mathcal{G}_{\emptyset}(\boldsymbol{\mathcal{X}})$ and $\mathcal{\hat G}_{\hat \emptyset}(\boldsymbol{\mathcal{X}})$, such that the corresponding events $\mathcal{E}_{I_n}^{(k,\epsilon)}$ and $\mathcal{E}_{J_n}^{(k,\epsilon)}$ defined as above using  $\mathcal{G}_{\emptyset}^{(k)}$ and $\mathcal{\hat G}_{\hat \emptyset}^{(k)}$, respectively, satisfy 
$$
P\left( \left. \mathcal{E}_{I_n}^{(k,\epsilon)}\cap \mathcal{E}_{J_n}^{(k,\epsilon)} \right| \mathscr{A}_n \right) \xrightarrow{P} 1, \qquad n \to \infty.
$$
\end{itemize}
\end{defn}

Strong couplings exist for a very large class of sparse random graph models, including: the Erd\H os-R\'enyi model, the Chung-Lu model (given expected degrees model), the Norros-Reittu model (Poissonian random graph), the generalized random graph, the configuration model, the stochastic block-model, and even preferential attachment graphs. We refer the reader to \cite{Olvera_21} for a more detailed discussion on the connection between strong couplings and various notions of local weak convergence for random graphs. We also mention that the notion of a strong coupling extends to the study of dense graphs, whose limit object is an non-locally finite graph, but those fall outside the scope of this paper. 

In addition to the existence of a strong coupling, we need to impose a technical condition on the external signals ensuring that individuals with similar vertex attributes  receive external signals that have similar distributions. The precise condition requires that we define a metric on the vertex attribute space $\mathcal{S}$, which we assume to take the form:
$$
\rho(\mathbf{x},  \mathbf{\tilde x}) = \rho'(\mathbf{a}, \mathbf{\tilde a}) + |q - \tilde q| + |s - \tilde s| + |d^- - \tilde d^-| + \sum_{j=1}^\infty \left| c_j 1(d^- \geq j) - \tilde c_j 1(\tilde d^- \geq j) \right|,
$$
where $\rho'$ is the metric used on $\mathcal{S}'$. The condition on the external signals correspond to a Lipschitz condition under the Wasserstein metric:
$$
d_1(\mu,\nu) = \inf\left\{ E[|X - Y|]: \text{law}(X) = \mu, \, \text{law}(Y) = \nu \right\}.
$$
In the sequel, we use $\nu(\mathbf{x}_i)$ to denote the common probability measure of the sequence $\{ W_i^{(k)}: k \geq 1\}$.

We are now ready to present our main result regarding the convergence of the typical stationary opinion on directed random graphs having local weak limits; $\Rightarrow$ denotes weak convergence on $\mathbb{R}$. 

\begin{theorem}\label{thm:treelim}
Suppose the graph sequence $\{G(V_n, E_n; \mathscr{A}_n): n \geq 1\}$ admits a strong coupling with a marked directed graph $\mathcal{G}(\boldsymbol{\mathcal{X}})$. In addition, assume $d>0$ and suppose the conditional distribution $\nu(\cdot)$ of the external signals given the vertex marks satisfies
$$
d_1(\nu(\mathbf{x}), \nu(\mathbf{\tilde x})) \leq K \rho(\mathbf{x}, \mathbf{\tilde x}),
\quad\text{for all } \mathbf{x},\mathbf{\tilde x}\in\mathcal{S}
$$
for some constant $K < \infty$. Let $\{ \mathbf{R}^{(k)}: k \geq 0\}$ denote the vectors of opinions on the graph $G(V_n, E_n; \mathscr{A}_n)$, and let $I_n$ be uniformly chosen in $V_n$. 
\begin{enumerate}
\item For any fixed $k \geq 0$ there exists a sequence of random variables $\{ \mathcal{R}^{(r)}: 0 \leq r \leq k\}$, whose distribution does not depend on $\mathscr{A}_n$, such that for any bounded and continuous function $f:[-1,1] \to \mathbb{R}$, 
$$
\max_{0 \leq r \leq k} E\left[ \left. \left| R_{I_n}^{(r)} - \mathcal{R}^{(r)} \right| \right| \mathscr{A}_n \right] \xrightarrow{P} 0 \qquad \text{and} \qquad \frac{1}{n} \sum_{i=1}^{n} f( R_i^{(k)}) \xrightarrow{P} E[ f( \mathcal{R}^{(k)}) ], \qquad n \to \infty.
$$

\item Let $\mathbf{R}$ be a vector having the stationary distribution of $\{ \mathbf{R}^{(k)}: k \geq 0\}$.  Then, there exists a random variable $\mathcal{R}^*$, whose distribution does not depend on $\mathscr{A}_n$, such that for any bounded and continuous function $f:[-1,1] \to \mathbb{R}$, 
$$
E\left[ \left. \left| R_{I_n} - \mathcal{R}^* \right| \right| \mathscr{A}_n \right] \xrightarrow{P} 0 \qquad \text{and} \qquad \frac{1}{n} \sum_{i=1}^{n} f( R_i) \xrightarrow{P} E[ f( \mathcal{R}^*) ], \qquad n \to \infty.
$$
Moreover, $\mathcal{R}^{(k)} \Rightarrow \mathcal{R}^*$ as $k \to \infty$ at a geometric rate. 
\end{enumerate}
\end{theorem}

\begin{remarks}\
\begin{enumerate}[leftmargin=*]
\item The random variable $\mathcal{R}^*$ represents the ``typical" opinion on any sufficiently large graph $G$, assuming that $G$  was sampled from a random graph model admitting a strong coupling. 
\item The distribution of $\mathcal{R}^*$ can also be used to compute the proportion of individuals in $G$ having opinions in a specific range. 
\item As we will see in the following section, when the local limit $\mathcal{G}(\boldsymbol{\mathcal{X}})$ is a marked Galton-Watson process, a lot can be said about the distribution of $\mathcal{R}^*$. 
\item The statement of the theorem for fixed $k$ provides a description of the typical trajectory of the opinion process. The sequence $\{ \mathcal{R}_\emptyset^{(r)}: 0 \leq r \leq k\}$ corresponds to the evolution of the opinions of the root of the local limit $\mathcal{G}_\emptyset(\boldsymbol{\mathcal{X}})$. Similar results for interacting diffusion processes on undirected graphs whose local weak limit is a unimodular Galton-Watson process have been established in \cite{Kavita_etal_2019, Kavita_etal_2020}.
\end{enumerate}
\end{remarks}

\section{Explicit Characterizations}\label{sec:explicit_characterizations}

When the local limit $\mathcal{G}(\boldsymbol{\mathcal{X}})$ is a tree, as is the case for all the examples mentioned in the previous section, the random variable $\mathcal{R}^*$ becomes tractable enough to compute its moments and numerically estimate its distribution function. To give an explicit representation of $\mathcal{R}^*$ it is useful to label nodes in a tree according to their ancestry from the root $\emptyset$. Specifically, a node in generation $k$ (at distance $k$ from the root) will have a label of the form $\mathbf{i} = (i_1, \dots, i_k) \in \mathbb{N}_+^k$, with the convention that $\mathbb{N}_+^0 = \{ \emptyset \}$. We use $(\mathbf{i},j) = (i_1, \dots, i_k, j)$ to denote the concatenation operation,  $|\mathbf{i}|$ to denote the generation of $\mathbf{i}$, and $\mathcal{U} = \bigcup_{k=0}^\infty \mathbb{N}_+^k$ to denote the set of finite sequences; to simplify the notation, for sequences of length one we simply write $\mathbf{i} = i$ instead of $\mathbf{i} = (i)$. A vertex $\mathbf{i}$ will have a full mark of the form:
$$
\boldsymbol{\mathcal{X}}_\mathbf{i} = (\boldsymbol{\mathcal{A}}_\mathbf{i}, \mathcal{Q}_\mathbf{i}, \mathcal{S}_\mathbf{i}, \mathcal{N}_\mathbf{i}, \mathcal{C}_{(\mathbf{i},1)} 1( \mathcal{N}_\mathbf{i}  \geq 1), \mathcal{C}_{(\mathbf{i},2)} 1( \mathcal{N}_\mathbf{i}  \geq 2), \dots) \in \mathcal{S},
$$
and each node $\mathbf{i}$ has a sequence of random signals $\{ \mathcal{W}_\mathbf{i}^{(k)}: k \geq 0\}$ that is conditionally i.i.d.~given $\boldsymbol{\mathcal{X}}_\mathbf{i} = \mathbf{x}$, with common distribution $\nu(\mathbf{x})$.

We can now explicitly write $\mathcal{R}^*$, as the following result shows.

\begin{prop}\label{prop:explicit}
Suppose that the assumptions of Theorem~\ref{thm:treelim} hold with $\mathcal{G}(\boldsymbol{\mathcal{X}})$ a tree, then
\begin{align}\label{eq:sol_memory}
    \mathcal{R}^*: =  \sum_{s=0}^{\infty} \sum_{l=0}^{s}  \sum_{|\mathbf{j}| = l} \Pi_{\mathbf{j}} a_{l,s} \mathcal{W}_{\mathbf{j}}^{(s)},
\end{align}
where $a_{l,s} = \binom{s}{l}(1-c-d)^{s-l}$, and $\Pi_{\mathbf{j}}$ is recursively defined by $\Pi_{(\mathbf{i}, j)} = \Pi_{\mathbf{i}}\mathcal{C}_{(\mathbf{i},j)}$ with $\Pi_\emptyset = 1$.
\end{prop}

\begin{remark} \label{rem:SFPE}
If $c+d = 1$, the autoregressive term in recursion \eqref{eq:OpinionRec} disappears, which we will refer to as the no-memory case, and $\mathcal{R}^*$ reduces to:
\begin{align} \label{eq:sol_no_memory}
    \mathcal{R}^*: =  \sum_{s=0}^{\infty}   \sum_{|\mathbf{j}| = s} \Pi_{\mathbf{j}} \mathcal{W}_{\mathbf{j}}^{(s)} =: \sum_{j=1}^{\mathcal{N}_\emptyset} \mathcal{C}_j \mathcal{R}_j + \mathcal{W}_\emptyset^{(0)},
\end{align}
with
\begin{equation} \label{eq:SpecialSol}
\mathcal{R}_j = \sum_{s=1}^\infty \sum_{|\mathbf{i}| = s-1}  \frac{\Pi_{(j,\mathbf{i})}}{\mathcal{C}_j} \mathcal{W}_{(j,\mathbf{i})}^{(s)},
\end{equation}
and the convention that $\Pi_{(j,\mathbf{i})}/\mathcal{C}_j \equiv 1$ if $\mathcal{C}_j  = 0$. Moreover, if $\mathcal{G}(\boldsymbol{\mathcal{X}})$ is a marked Galton-Watson process (as it is for all the models mentioned above with the exception of preferential attachment graphs), the $\{\mathcal{R}_i\}$ are i.i.d.~copies of the so-called special endogenous solution to the branching stochastic fixed-point equation (SFPE):
$$
\mathcal{R} \eqlaw \sum_{j=1}^{\mathcal{N}_1} \mathcal{C}_{(1,j)} \mathcal{R}_j + \mathcal{W}_1^{(1)},
$$
where the $\{\mathcal{R}_j\}$ are i.i.d.~copies of $\mathcal{R}$, independent of $(\mathcal{N}_1, \mathcal{W}_1^{(1)}, \mathcal{C}_{(1,1)}, \mathcal{C}_{(1,2)}, \dots)$. Having this SFPE characterization makes the computation of moments easier and their expressions simpler. This linear SFPE is known in the literature as the smoothing transform, and the study of its solutions dates back to the 70s. A classical survey on branching SFPEs and the notion of endogeny is \cite{Aldo_Band_05}. The full characterization of all the solutions to the smoothing transform was done in \cite{Als_Big_Mei_12,Als_Mei_12}, and it is also there where  \eqref{eq:SpecialSol} was given the name ``special endogenous solution''. 
\end{remark}

From now on we will assume that the local limit $\mathcal{G}(\boldsymbol{\mathcal{X}})$ of the random graph sequence $\{G(V_n, E_n; \mathscr{A}_n): n \geq 1\}$ is a (delayed) marked Galton-Watson tree, i.e., the node marks $\{ \boldsymbol{\mathcal{X}}_\mathbf{i}: \mathbf{i} \in \mathcal{U} \}$ are independent of each other, with $\{ \boldsymbol{\mathcal{X}}_\mathbf{i}: \mathbf{i} \in \mathcal{U}, \mathbf{i} \neq \emptyset \}$ i.i.d. The qualifier ``delayed" refers to the possibility of $\boldsymbol{\mathcal{X}}_\emptyset$ having a different distribution than all other node marks, which is an important feature of most of the random graph models for which strong couplings exist. The different distribution is due to the {\em size-bias} encountered during the exploration of the graph when sampling vertices that are known to be inbound neighbors of another vertex (the Erd\H os-R\'enyi model and all regular graphs are exempt from size-bias).

To emphasize that the local limit is assumed from now on to be a tree, we will denote the marked tree by $\mathcal{T}(\boldsymbol{\mathcal{X}})$. In the next subsections we will use the explicit representation \eqref{eq:sol_memory} and the properties of $\mathcal{T}(\boldsymbol{\mathcal{X}})$ to explain when consensus exists, when polarization occurs, what the presence of stubborn agents can do, and what are the parameters in our model that explain these phenomena. 
\subsection{Consensus}\label{sec:consensus}

The existence of consensus for opinion dynamics on networks has been the main focus of much of the early work \cite{abelson1964mathematical, degroot1974reaching, french1956formal}. For a fixed graph $G$, consensus means that all the components in the vector of stationary opinions $\mathbf{R} = (R_1, \dots, R_{|V|})$ are equal to each other. In the context of a sequence of random graphs $\{G(V_n, E_n; \mathscr{A}_n): n \geq 1\}$ having a local weak limit, consensus would mean that the random variable $R_{I_n}$ defined in \eqref{eq:LWC} is a constant, which under the assumptions of Theorem~\ref{thm:treelim} would be  equivalent to the random variable $\mathcal{R}^*$ being a constant. In other words, we identify consensus with $\var(\mathcal{R}^*)$ being close to zero. 

Before we explain how consensus can occur in our opinion model, note that on a tree its dynamics are described by the recursion:
$$
\mathcal{R}_\mathbf{i}^{(k+1)} = \sum_{j=1}^{\mathcal{N}_\mathbf{i}} \mathcal{C}_{(\mathbf{i}, j)} \mathcal{R}_{(\mathbf{i},j)}^{(k)} + \mathcal{W}_\mathbf{i}^{(k)} + (1-c-d) \mathcal{R}_\mathbf{i}^{(k)}, \qquad \mathbf{i} \in \mathcal{T}(\boldsymbol{\mathcal{X}}), \qquad k \geq 0,
$$
where $\mathcal{R}_\mathbf{i}^{(k)}$ is the opinion of node $\mathbf{i}$ at time $k$. We will also assume throughout this section that the external signals take the form
$$
\mathcal{W}_\mathbf{i}^{(k)} = d \mathcal{Z}_\mathbf{i}^{(k)}  + \mathcal{Q}_\mathbf{i} \left( c - \sum_{j=1}^{\mathcal{N}_\mathbf{i}} \mathcal{C}_{(\mathbf{i},j)} \right),
$$
with $\{ \mathcal{Z}_\mathbf{i}^{(k)}: k \geq 0, \mathbf{i} \in \mathcal{U} \} \subseteq [-1,1]$ conditionally independent given $\mathcal{T}(\boldsymbol{\mathcal{X}})$, and $\{ \mathcal{Z}_\mathbf{i}^{(k)}: k \geq 0 \}$ conditionally i.i.d.~given $\boldsymbol{\mathcal{X}}_\mathbf{i}$. In words, for nodes having a positive in-degree $\mathcal{N}_\mathbf{i}$, the media contributes a proportion $d$ to that node's opinion, while its neighbors contribute a proportion $c$. The effect of the neighbors can be understood as promoting consensus, while the media signals introduce noise. As $d$ grows, the media's influence on the formation of opinions increases, making the stationary opinion mimic the media. 

The result below gives explicit formulas for the mean and variance of $\mathcal{R}^*$ which explain when consensus occurs. We only state the formulas for the no-memory case since the general case is qualitatively the same. We also assume for most of the results in Section~\ref{sec:explicit_characterizations} that all vertices have at least one neighbor, since this significantly simplifies the exposition. The formulas for the general case are included in the appendix. 

\begin{prop}\label{prop:consensus}
Suppose the assumptions of Theorems~\ref{thm:treelim} hold, with a (delayed) marked Galton-Watson tree  $\mathcal{T}(\boldsymbol{\mathcal{X}})$ as the local limit. Then, if $P(\mathcal{N}_\emptyset > 0) = P(\mathcal{N}_1 > 0) = 1$, and $c+d=1$, the mean and variance of $\mathcal{R}^*$ are given by:
    \begin{align*}
    E\left[ \mathcal{R}^* \right] &=d
    E\left[ \mathcal{Z}_\emptyset \right] + cE[\mathcal{Z}_1] ,\\
    \var(\mathcal{R}^*)  &= d^2\var(\mathcal{Z}_\emptyset)  +  \frac{ \rho_2^*d^2}{1-\rho_2}
    \var(\mathcal{Z}_1), 
    \end{align*}
where $\rho_2^* = E\left[\sum_{i = 1}^{\mathcal{N}_\emptyset} \mathcal{C}_i^2\right]$ and $\rho_2 = E\left[\sum_{i = 1}^{\mathcal{N}_1} \mathcal{C}_{(1,i)}^2\right]$.
\end{prop}

\begin{remarks}\
\begin{enumerate}[leftmargin=*]
\item The expression for the variance states that consensus can happen when $d \mathcal{Z}_\emptyset$ and $d \mathcal{Z}_1$ have small variances.

\item When $d$ is small, the consensus is attained through the averaging of the opinions of each node's neighbors, which in the long run leads to everybody on the limiting network (connected since it is assumed to be a tree) having the same opinion. A small $d$ in this case simply means that the media signals play a small role in the formation of opinions, regardless of the nature of the signals. 

\item  Even if $d$ is not very small, consensus can still occur if the media signals $\{\mathcal{Z}_\mathbf{i}^{(k)}: k \geq 0\}$ themselves have a small variance, in other words, if everyone in the network is listening to essentially the same opinion. From a modeling point of view, this case could be used to model a state-controlled media that is highly trusted by the citizens. 

\item For the no-memory case ($c+d = 1$), the characterization of $\mathcal{R}^*$ in terms of the special endogenous solution to the smoothing transform (see Remark~\ref{rem:SFPE}), shows that the stationary opinion cannot be a single point, since no Dirac measure solves the SFPE. 
\end{enumerate}
\end{remarks}

To illustrate the results in this section, we also performed a simulation of an Erd\H os-R\'enyi graph with $n=1000$ vertices and edge probability $p = 0.03$. Vertices with zero in-degree are extremely rare, and have no special attributes if they occur. The vector of weights $(C(i,1), \dots, C(i,D_i^-))$ gives equal weight $c/D_i^-$ to each inbound neighbor. Everyone's initial opinion $R_i^{(0)}$ is chosen to be uniform on $\{-1,1\}$. We then chose a large value of $k$ and iterate \eqref{eq:OpinionRec} to approximate a realization of the stationary vector $\mathbf{R} = (R_1, \dots, R_n)$. We do this for two different choices of $c,d \geq 0$, one with $c+d < 1$ (general case) and one with $c+d = 1$ (no-memory case), with the latter parameters chosen proportionally to the weights given to the neighbors and to the media in the general case (i.e., the contribution of the neighbors is $c/(c+d)$ and that of the media is $d/(c+d)$). The media signals $\{ Z_i^{(k)}: k \geq 0, i \in V_n\}$ are chosen to be i.i.d.~and independent of the vertex attributes, which means that the internal opinions $\{Q_i: i \in V_n\}$ play essentially no role. The plots in Figure~\ref{fig:consensus_media_small_var}, \ref{fig:consensus_small_d}, and \ref{fig:consensus_not_in_the_middle} plot the histograms computed from the simulated values of the components of $\mathbf{R}$, which by Theorem~\ref{thm:treelim} are consistent estimators for the distribution of $\mathcal{R}^*$. 

Figure~\ref{fig:consensus_media_small_var} shows the case where $d$ is ``large", i.e., close to one, and media signals that have very small variance (e.g., highly trusted state-controlled media). As predicted by Proposition~\ref{prop:consensus}, the stationary opinion is concentrated around the mean of the media signals, in this case, zero. 

\begin{figure}
    \centering
    \includegraphics[scale = 0.6, trim={2.3cm 5cm 2.5cm 5cm},clip]{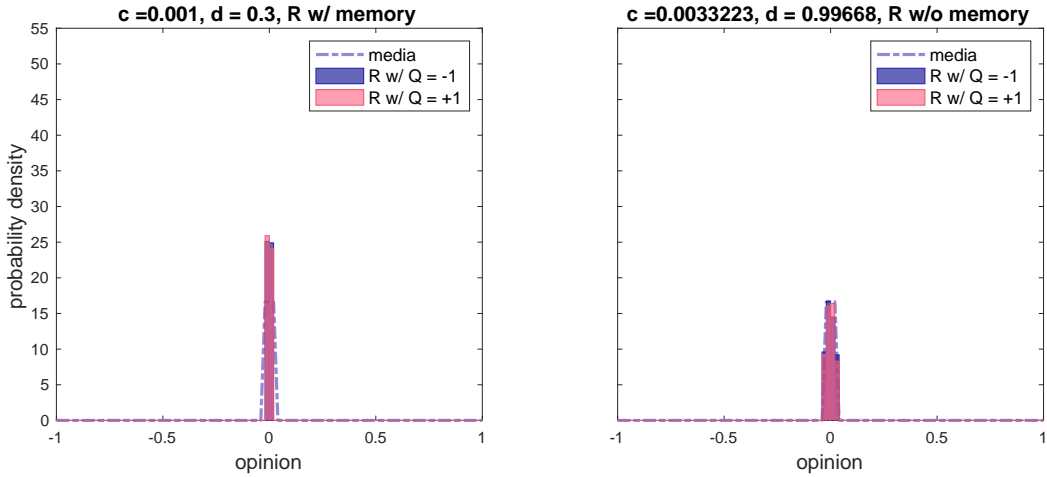}
     \caption{Empirical distribution of opinions in an Erd\H os-R\'enyi graph with $G(1000, 0.03)$ with internal opinion $Q_i \sim \text{Unif}(-1,1)$ and media signals $Z_{i}^{(k)}\sim \text{Unif}(-0.03,0.03)$, independent of the vertex attributes. }
     \label{fig:consensus_media_small_var}
\end{figure}

Figure~\ref{fig:consensus_small_d} and Figure~\ref{fig:consensus_not_in_the_middle} illustrate the case where $d$ is close to zero, i.e., people trust their neighbors more than they trust the media, and the averaging of opinions eventually leads to consensus, regardless of the shape of the media. Figure~\ref{fig:consensus_small_d} depicts media signals that are polarized, since they send signals in $\{-1,1\}$, however, the signals are unrelated to the individuals' attributes, so everyone is exposed to both kinds of extreme signals. Figure~\ref{fig:consensus_not_in_the_middle} shows how the mean of the media signals determines where the consensus is attained. We note that in these experiments, the no-memory case and the general case behave in a very similar way. 

\begin{figure}
    \centering
   \includegraphics[scale = 0.6, trim={2.3cm 5cm 2.5cm 5cm},clip]{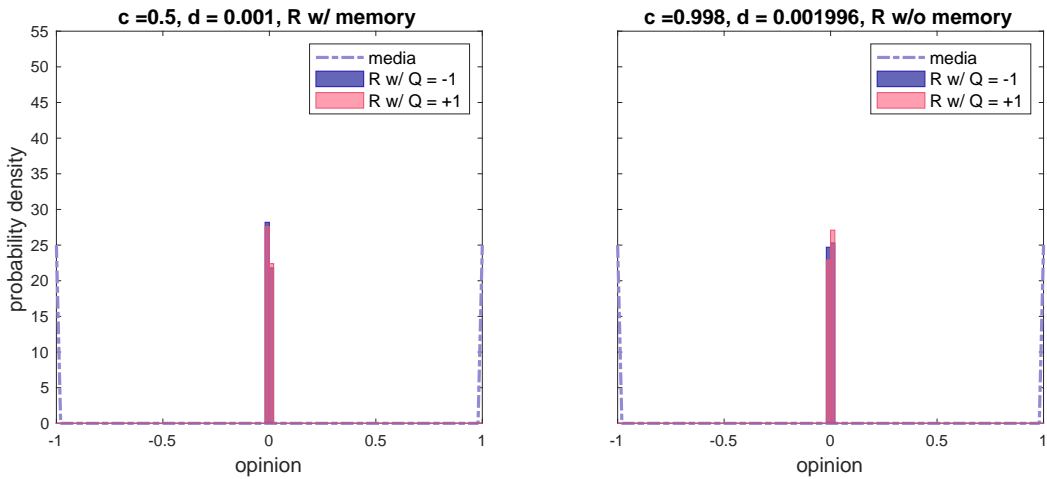}
    \caption{Empirical distribution of opinions in an Erd\H os-R\'enyi graph $G(1000, 0.03)$ with internal opinion $Q_i \sim \text{Unif}(-1,1)$, and media signals $Z_{i}^{(k)}\sim \text{Unif}\{-1,1\}$, independent of the vertex attributes. }
    \label{fig:consensus_small_d}
\end{figure}

\begin{figure}
    \centering
  \includegraphics[scale = 0.6, trim={2.3cm 5cm 2.5cm 5cm},clip]{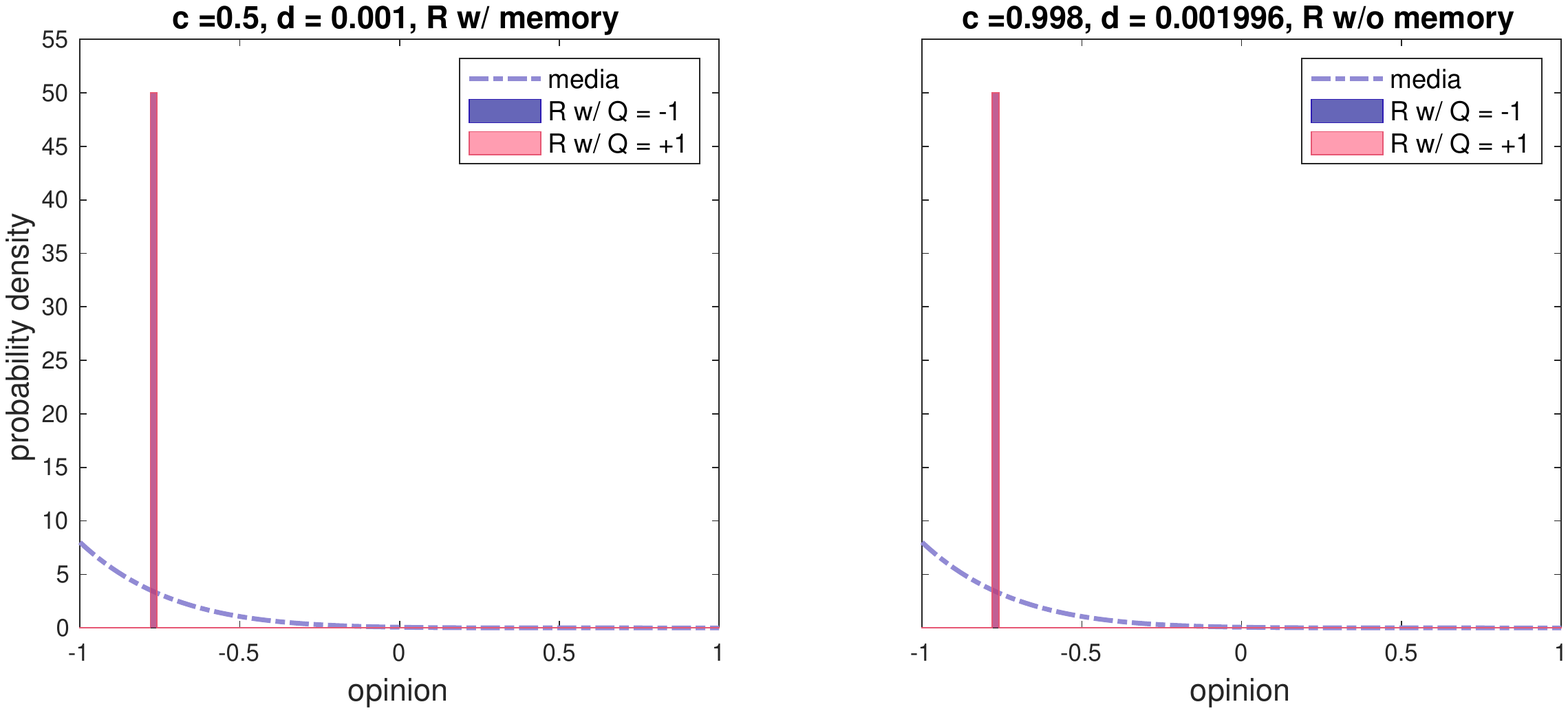}
    \caption{Empirical distribution of opinions in an Erd\H os-R\'enyi graph $G(1000, 0.03)$ with internal opinion $Q_i \sim \text{Unif}(-1,1)$, and media signals $Z_{i}^{(k)}\sim -1+2\text{Beta}(1,8)$, independent of the vertex attributes. }
    \label{fig:consensus_not_in_the_middle}
\end{figure}

\subsection{Polarization}

We now move on to explain how our model can also result in opinions that are polarized, in the sense that opinions in the network are split into two groups with opposing beliefs on the topic of interest. Our model produces polarization by allowing individuals in the network to listen to media signals that are aligned with their internal beliefs, which is consistent with the phenomenon that social scientists refer to as \emph{selective exposure} \cite{festinger1957cognitive}. Mathematically, polarization means that the distribution of $\mathcal{R}^*$ is bimodal and $\var(\mathcal{R}^*)$ is not small, with the two modes representing the two groups that attain some level of consensus within themselves but not with each other.  Our model is designed to incorporate each individual's profile through the vertex attributes, so we use the vertex attributes as a basis for separating individuals into distinct groups. In this setting, polarization would imply that similar individuals are likely to end up in the same group, which translates into $\var(\mathcal{R}^* | \boldsymbol{\mathcal{X}}_\emptyset)$ being small. The precise result is given below.

\begin{prop}\label{prop:polarization}
Suppose the assumptions of Theorems~\ref{thm:treelim} hold, with a (delayed) marked Galton-Watson tree  $\mathcal{T}(\boldsymbol{\mathcal{X}})$ as the local limit. Then, if $P(\mathcal{N}_\emptyset > 0) = P(\mathcal{N}_1 > 0) = 1$, and $c+d=1$, the conditional mean and variance of $\mathcal{R}^*$ given the vertex attribute $\boldsymbol{\mathcal{X}}_\emptyset$ are given by:
\begin{align*}
E\left[ \mathcal{R}^* | \boldsymbol{\mathcal{X}}_\emptyset \right] & = dE\left[\mathcal{Z}_\emptyset|\boldsymbol{\mathcal{X}}_\emptyset \right]+ cE[\mathcal{Z}_1], \\
\var(\mathcal{R}^* |\boldsymbol{\mathcal{X}}_\emptyset )
&=   d^2 \var\left(\left.\mathcal{Z}_\emptyset \right|\boldsymbol{\mathcal{X}}_\emptyset\right)+\frac{d^2}{1 - \rho_2}\sum_{i = 1}^{\mathcal{N}_{\emptyset}}\mathcal{C}_{i}^2\var(\mathcal{Z}_1),
\end{align*}
where $\rho_2 = E\left[\sum_{i = 1}^{\mathcal{N}_1}C_{(1,i)}^2\right]$.
\end{prop}


\begin{remarks} \label{rem:Polarization}\
\begin{enumerate}[leftmargin=*]
\item The unconditional terms in both the mean and variance correspond to the neighbors' contribution, which in our model are equally likely to hold opinions of any type. As we can see from the expression for the conditional mean, the term $cE[\mathcal{Z}_1]$ reflects the entire network's mean opinion, while the term $d E[\mathcal{Z}_\emptyset | \boldsymbol{\mathcal{X}}_\emptyset]$ can shift the mean opinion towards one end of the spectrum depending on the individual's profile.

\item Note that the variance of the stationary opinion over the entire network corresponds to $\var(\mathcal{R}^*)$, which by Proposition~\ref{prop:consensus} is given by:
$$
\var(\mathcal{R}^*) = d^2 \var(\mathcal{Z}_\emptyset)  +  \frac{ \rho_2^*d^2}{1-\rho_2} \var(\mathcal{Z}_1),
$$
which implies that
$$
E\left[ \var(\mathcal{R}^*)  - \var(\mathcal{R}^* |\boldsymbol{\mathcal{X}}_\emptyset ) \right]  = d^2 \var\left( E \left[ \left.\mathcal{Z}_\emptyset \right|\boldsymbol{\mathcal{X}}_\emptyset\right]  \right).
$$
It follows that polarization will occur whenever $d$ is not too small and $\var\left( E \left[ \left.\mathcal{Z}_\emptyset \right|\boldsymbol{\mathcal{X}}_\emptyset\right]  \right)$ is large. 

\item A straightforward way to introduce selective exposure is to allow the media signals $\{ \mathcal{Z}_\mathbf{i}^{(k)}: k \geq 0 \}$ to depend on the internal opinion $\mathcal{Q}_\mathbf{i}$ in such a way that if $\mathcal{Q}_\mathbf{i} > 0$, then the media signals $\{ \mathcal{Z}_\mathbf{i}^{(k)}: k \geq 0 \}$ tend to be close to $+1$, while for $\mathcal{Q}_\mathbf{i} < 0$ the media signals tend to be close to $-1$.  Our numerical examples below illustrate different levels of correlation between an individual's internal opinion and the media signals they receive. 

\item We point out that the confirmation bias that our results in this section can cover is limited to the media signals through selective exposure, in other words, it does not apply to connections between individuals. The more realistic setting where individuals with opposing views on a topic are less likely to be connected can be incorporated into the model through the underlying graph, say by assuming that $G(V_n, E_n; \mathscr{A}_n)$ is a stochastic block model (SBM). In a SBM one can modulate the connectivity between individuals based on their internal opinion, which would conceivably result in even higher levels of polarization. The SBM is in fact covered by Theorem~\ref{thm:treelim}, however, its local weak limit is a marked multi-type Galton-Watson tree, rather than the single-type considered in Section~\ref{sec:explicit_characterizations}. We will give the characterization of our opinion model with full confirmation bias in a future follow up work. 

\end{enumerate}
\end{remarks}

%

We end this section with a numerical example to illustrate a polarized stationary opinion. The general simulation setup is the same from Section~\ref{sec:consensus}, with the only difference being that now we choose the media signals $\{ Z_i^{(k)}: k \geq 0\}$ to depend on the internal opinion $Q_i$ for each $i \in V_n$.  In Figure~\ref{fig:polarization_biased_media} we chose $Q_i$ to be equally likely $+1$ or $-1$. The media signals are chosen in alignment with the internal opinion, as follows:
\begin{align*}
(Z_i^{(k)} | Q_i = +1) &\sim -1+2\text{Beta}(8,1), \\
(Z_i^{(k)} | Q_i = -1) &\sim -1+2\text{Beta}(1,8),
\end{align*}
for each $i \in V_n$. The corresponding population variance is $\var(\mathcal{R}^*) = 0.1484$, and the conditional means and variances are:
\begin{align*}
E\left[ \mathcal{R}^* | \mathcal{Q}_\emptyset = -1\right] & \approx -0.3684, \qquad E\left[ \mathcal{R}^* | \mathcal{Q}_\emptyset = 1\right]  \approx 0.3684,\\
\var(\mathcal{R}^* | \mathcal{Q}_\emptyset = -1) & = 0.0095, \qquad \var(\mathcal{R}^* | \mathcal{Q}_\emptyset = 1) = 0.0095.
\end{align*}

\begin{figure}[ht]
    \centering
    \includegraphics[scale = 0.6, trim={2.3cm 5cm 2.5cm 5cm},clip]{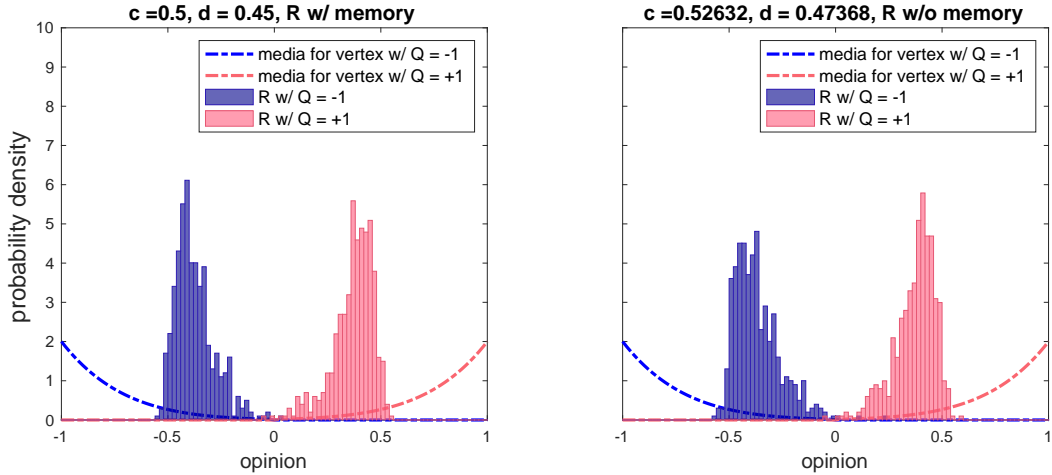}
    \caption{Empirical distribution of opinions in an Erd\H os-R\'enyi graph $G(1000, 0.03)$ with internal opinion $Q_i \sim \text{Unif}\{-1,1\}$. The media signals are biased towards one's internal opinion by setting $(Z_{i}^{(k)} | Q_i = +1) \sim -1+2\text{Beta}(8,1)$ and $(Z_i^{(k)}| Q_i = -1) \sim -1 + 2\text{Beta}(1,8)$.}\label{fig:polarization_biased_media}
\end{figure}

\subsection{Stubborn agents}

The numerical examples in the previous two sections were designed to have an average in-degree/out-degree of 30, which means that in practice there are very few, if any, vertices with zero in-degree. However, it is easy to incorporate a positive fraction of vertices with no in-degree by choosing an inhomogeneous random graph model, e.g., a configuration model, a Chung-Lu model, a Norros-Reittu model, or a generalized random graph. These models have parameters that regulate the in-degree/out-degree distribution of the vertices in the graphs they produce, so simply choosing a distribution where zero in-degree occurs with positive probability can give the desired result. Once the underlying graph is chosen to produce a positive fraction of zero in-degree vertices, we can model the presence of disruptive ``stubborn agents" or ``bots" by giving them an internal opinion that is extreme and choosing their media signals to copy their internal opinion (so that they never change their opinion). Our setting even allows us to have both ordinary stubborn agents (who may espouse typical internal opinions and/or still be influenced by the media) and bots by setting the vertex attribute $S_i = 1$ for the bots and $S_i = 0$ for all other vertices. 

We point out that modeling the presence of stubborn agents as we described above is not equivalent to modeling a full confirmation bias, since it only allows us to separate the behavior of certain individuals based on whether they have zero in-degree/out-degree or not, but it does not allow us to modulate the connectivity between regular vertices based on their internal opinions. As explained in Remark~\ref{rem:Polarization}(4), the easiest way to model individual-level confirmation bias would be to assume $G(V_n, E_n; \mathscr{A}_n)$ is a SBM. 

Given that the setting of Section~\ref{sec:explicit_characterizations} excludes the SBM, bots cannot target individuals based on their internal opinions, which reduces their effectiveness for polarizing the stationary opinion distribution. To see why this is the case, suppose one introduces bots that are equally likely to have opinions in $\{-1,1\}$. Then, each regular individual will be equally likely to have a $-1$ bot as an inbound neighbor than a $+1$ bot, hence balancing the effect of bots on the overall network. However, if the bots are unbalanced, in the sense that there is a greater proportion of bots with opinion $+1$ than bots with opinion $-1$, or vice versa, then they have the power to shift the entire stationary opinion distribution in their direction. Figure~\ref{fig:unbalanced_stubborn} depicts exactly this situation on a graph that is very similar to the one used in Figure~\ref{fig:polarization_biased_media}, with the exception of the presence of bots. 

\begin{figure}[ht]
    \centering
    \includegraphics[scale = 0.6, trim={2.3cm 5cm 2.5cm 5cm},clip]{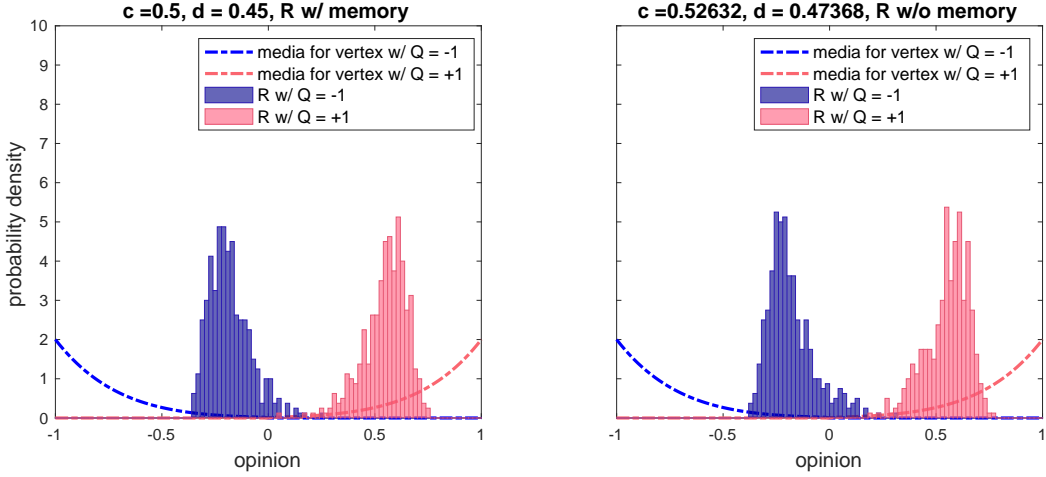}
    \caption{Empirical distribution of opinions in an inhomogeneous random digraph with 800 regular vertices that are connected using an Erd\H os-R\'enyi graph $G(800, 0.03)$ with internal opinions $(Q_i | S_i = 0) \sim \text{Unif}\{-1,1\}$. Their media signals are biased towards one's internal opinion by setting $(Z_i^{(k)}| Q_i = +1, S_i=0)\sim -1+2\text{Beta}(8,1)$ and $(Z_i^{(k)}| Q_i = -1, S_i=0)\sim -1+2\text{Beta}(1,8)$. In addition to the 800 regular vertices, the graph has 200 bots that have zero in-degree, internal opinions $(Q_i|S_i = 1) =1$, and media signals $(Z_i^{(k)}| S_i = 1) = 1$ for all $k \geq 0$. Bots connect to regular vertices using independent Bernoulli$(0.03)$ trials. } \label{fig:unbalanced_stubborn}
\end{figure}
\subsection{Memory vs no-memory} \label{sec:memory}

All the numerical experiments that have been presented in the previous sections have included histograms for the stationary opinion on the network for both the general recursion $(c+d < 1)$ and the no-memory recursion $(c+d = 1)$. In those examples the difference between the two cases is very subtle. To better understand the impact of the autoregressive term in the general case, it is useful to consider scenarios where the variance of the stationary opinion, $\var(\mathcal{R}^*)$ is relatively large.   Intuitively, the autoregressive term is slowing down the speed at which individuals update their opinions, and therefore, we can expect individual trajectories to be smoother, which in turn would imply a less variable stationary opinion. This is indeed the case, as Proposition~\ref{prop:memory} shows. 

\begin{prop}\label{prop:memory}
Suppose the assumptions of Theorems~\ref{thm:treelim} hold, with a (delayed) marked Galton-Watson tree  $\mathcal{T}(\boldsymbol{\mathcal{X}})$ as the local limit. Let $\mathcal{M}^*$ and $\mathcal{R}^*$ denote the limiting opinion for the general recursion
\begin{equation} \label{eq:recursion_on_graph} 
\mathcal{M}_\mathbf{i}^{(k+1)} = \sum_{j=1}^{\mathcal{N}_\mathbf{i}} \mathcal{C}_{(\mathbf{i}, j)} \mathcal{M}_{(\mathbf{i},j)}^{(k)} + \ \mathcal{W}_\mathbf{i}^{(k)} + (1-c-d) \mathcal{M}_\mathbf{i}^{(k)}, \qquad \mathbf{i} \in \mathcal{T}(\boldsymbol{\mathcal{X}}), \qquad k \geq 0,
\end{equation}
with $c,d > 0$ and $c+d <1$, and the no-memory recursion
\begin{equation} \label{eq:no_memory_recursion}
\mathcal{R}_\mathbf{i}^{(k+1)} = \sum_{j=1}^{\mathcal{N}_\mathbf{i}} \frac{\mathcal{C}_{(\mathbf{i}, j)}}{c+d} \,  \mathcal{R}_{(\mathbf{i},j)}^{(k)} + \frac{\mathcal{W}_\mathbf{i}^{(k)}}{c+d} , \qquad \mathbf{i} \in \mathcal{T}(\boldsymbol{\mathcal{X}}), \qquad k \geq 0,
\end{equation}
Then, if $P(\mathcal{N}_\emptyset > 0) = P(\mathcal{N}_1 > 0) = 1$, we have
$$
\var(\mathcal{M}^*) \leq \var( \mathcal{R}^*).
$$
\end{prop}

We illustrate this phenomenon with a numerical example having non-negligible variance. In Figure~\ref{fig:memory_hist} we simulate again an Erd\H os-R\'enyi graph with $n = 1000$ vertices and edge probabilities $p = 0.03$, as in previous experiments. The media signals $\{Z_i^{(k)}: k \geq 0, i \in V_n\}$ were sampled uniformly on the interval $[-1,1]$, and independently of the internal opinions. The internal opinions have no impact on the dynamics of the model. The choice of $c$ and $d$ in the no-memory case was done proportionally to the values of $c$ and $d$ in the general case. There are few, if any, vertices with zero in-degree, and they have no special attributes. 

\begin{figure}
    \centering
    \includegraphics[scale = 0.6, trim={2.3cm 5cm 2.5cm 5cm},clip]{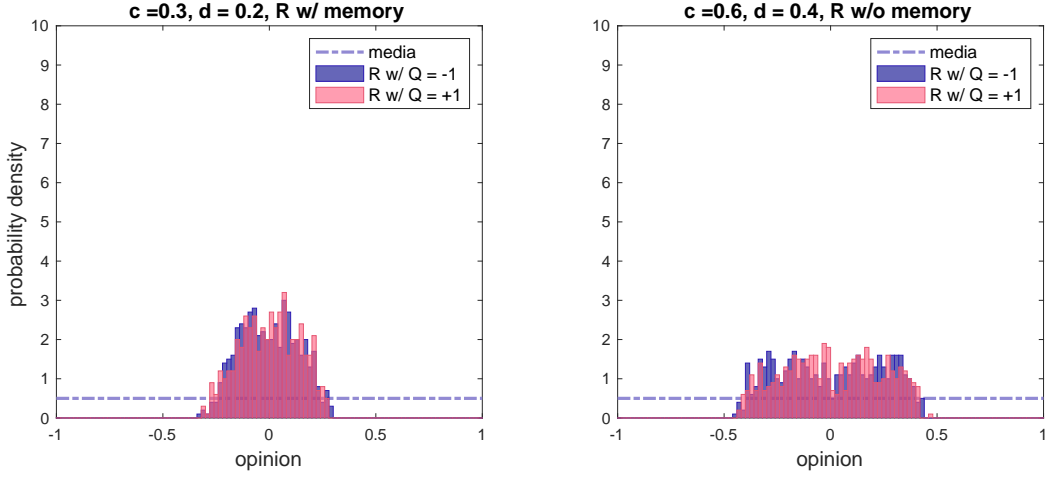}
    \caption{Empirical distribution of opinions in an Erd\H os-R\'enyi graph $G(1000, 0.03)$ with internal opinion $Q_i \sim \text{Unif}\{-1,1\}$ and media signals $Z_{i}^{(k)}\sim \text{Unif}(-1,1)$, independent of the vertex attributes.} \label{fig:memory_hist}
\end{figure} 

As mentioned earlier, the effect of the autoregressive term also makes the individual trajectories $\{ R_i^{(k)}: k \geq 0\}$ for each vertex $i \in V_n$, smoother. Figure~\ref{fig:memory_trajectory} shows the trajectories of two vertices, one starting from an initial state $R_i^{(0)} = -1$ and another starting from the initial state $R_j^{(0)} = +1$. The setting for the simulation is exactly the same as the one for Figure~\ref{fig:memory_hist}. As we can see, the opinions for the no-memory recursion~\eqref{eq:no_memory_recursion} appear to converge to stationarity in roughly one iteration, where they fluctuate around the interval $[-0.5, 0.5]$. The general recursion~\eqref{eq:recursion_on_graph}, on the other hand, takes about 10 iterations to reach stationarity, and once there it fluctuates around a smaller interval.  

\begin{figure}
    \centering
    \includegraphics[scale = 0.6, trim={2cm 7cm 2.5cm 7cm},clip]{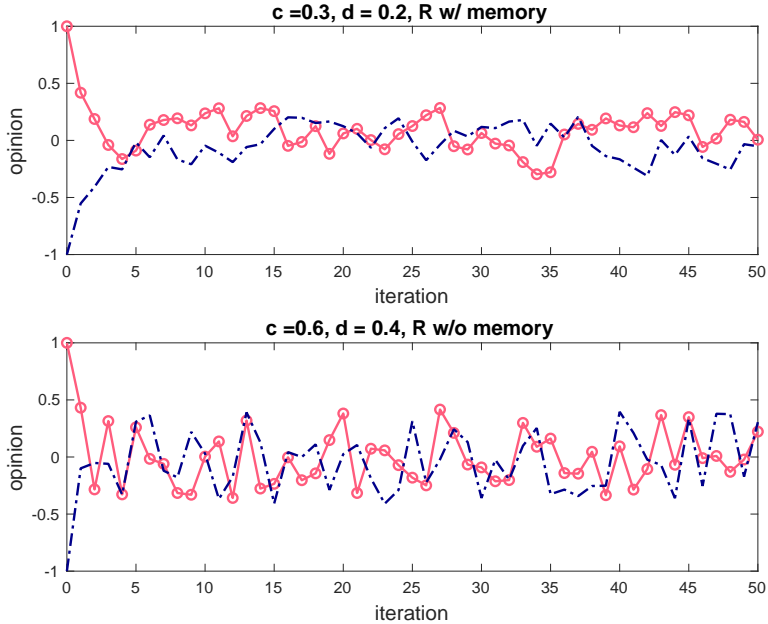}
    \caption{Trajectories of a typical vertex in an Erd\H os-R\'enyi graph $G(1000, 0.03)$ with internal opinion $Q_i \sim \text{Unif}\{-1,1\}$ and media signals $Z_{i}^{(k)}\sim \text{Unif}(-1,1)$, independent of the vertex attributes. The figure contains the trajectories of two distinct vertices, one initialized at $R_i^{(0)} = -1$ and another at $R_j^{(0)} = +1$.}     \label{fig:memory_trajectory}
\end{figure}

\section{Related Literature}\label{sec:literature}

The early work in opinion dynamics modeling dates back to the 1960s, when the first opinion-averaging models on strongly connected graphs appeared \cite{abelson1964mathematical, degroot1974reaching, french1956formal}. Among these models, the classical DeGroot model \cite{degroot1974reaching} has been the paradigm for discrete-time opinion dynamics on a continuous state space. The main question of interest in both \cite{degroot1974reaching}, and the latter work in \cite{berger1981necessary}, is the existence of consensus, interpreted as the existence of a single opinion that all individuals in the network will eventually agree on. 

The study of disagreement and/or polarization is more recent, starting with the work by Friedkin and Johnsen  \cite{friedkin1990social}, where they extended the DeGroot model to incorporate the concept of social prejudice to account for disagreements in social networks. Mathematically, the original Friedkin-Johnsen model is a special case of \eqref{eq:OpinionRec}, where $c+d = 1$ and the external signals $\{W_i^{(k)}: k \geq 0, i \in V\}$ are assumed to be of the form $W_i^{(k)} = dQ_i$ for all $k \geq 0$. That is, the model does not have any randomness but it does allow for vertex attributes to influence the evolution of opinions. The convergence of the Friedkin-Johnsen model under asynchronous updating was studied in \cite{parsegov2016novel, ravazzi2014ergodic}, and the question of consensus formation under the effect of increasing peer pressure has been investigated in \cite{semonsen2018opinion}. Empirical studies of the prevalence of persistent disagreement in society include \cite{baron1996social,dimaggio1996have}.

There has also been significant work on opinion dynamics models that incorporate some randomness in the updating rule. The random noise can be interpreted as media signals \cite{yang2017innovation}, individuals' free will \cite{pineda2009noisy}, environmental uncertainty \cite{zhao2016bounded}, etc. The noisy DeGroot and Friedkin-Johnsen models have been numerically examined in \cite{konovalchikova2020opinion, stern2021impact}, and theoretically in \cite{xiao2007distributed, yang2017innovation}. The model in \cite{xiao2007distributed} as well as Model II in \cite{yang2017innovation} correspond to taking the external signals $\{W_i^{(k)}: k \geq 0, i \in V\}$ i.i.d.~and independent of the underlying graph in \eqref{eq:OpinionRec}. The main question of interest in these last two works is the existence or non-existence of consensus, and their findings are consistent with ours. Although some of the references mentioned above do use random graphs in their numerical analyses, none of them use random graph techniques to characterize the stationary opinion distribution. 

Our work in this paper is also closely related to the problem of modeling the bias in the formation of opinions and how it can lead to polarization \cite{bakshy2015exposure,garimella2018political,lord1979biased,munro2002biased}. Theoretically, this type of bias has mostly been tackled by opinion models where during the updating step, a vertex only considers the opinions of neighbors that are close enough to their own current opinion. These types of models are generically known as bounded confidence models, with two of the most popular ones being the Deffuant–Weisbuch model \cite{carro2013role, deffuant2000mixing, weisbuch2004bounded} and the Hegselmann-Krause model \cite{canuto2012eulerian, hegselmann2002opinion,  mirtabatabaei2011opinion}. The analysis of bounded confidence models often relies on continuous-time dynamical systems, which have been shown to lead to consensus, fragmentation, and polarization, depending on their parameters. However, the predictions based on continuous dynamical systems do not always agree with more realistic simulation-based experiments. Another model worth mentioning due to its ability to predict polarization is the generalization of the DeGroot model proposed in \cite{dandekar2013biased}. There, the influence of neighboring opinions is determined by a bias parameter that produces non-linear updates, and that is reminiscent of preferential attachment rules. 

Our approach to modeling polarization in the current paper is very different from the works we just described, and is based on the introduction of selective exposure into the external signals $\{ W_i^{(k)}: k \geq 0, i \in V\}$, while preserving the linearity of the model. As we described in Remark~\ref{rem:Polarization}(4), we can model individual-level confirmation bias by selecting the underlying graph to prioritize connections between individuals holding similar internal opinions (e.g., by using a SBM), but we have left that analysis for future work. Another distinguishing feature of our work is the use of random graph theory to characterize the stationary opinion distribution on large complex networks, which to the best of our knowledge, has not been done in the existing literature on opinion dynamics. 

Before concluding our overview of opinion modeling on directed networks, we would like to mention some other directions that people have considered. The DeGroot and Friedkin-Johnsen models were generalized in \cite{rey2017evolution} to include more than one topic at a time. In \cite{mackin2019maximizing} the authors consider the problem of maximizing diversity in the opinions by strategically placing two disruptive stubborn agents on a fixed graph. The work in \cite{proskurnikov2015opinion} analyzes the DeGroot model on a signed graph, where consensus, disagreement, and polarization are all possible. In addition, the dynamics of the DeGroot model (often in continuous time) have been generalized to allow the weights given to neighbors to depend on the evolution of the opinion process itself. Although the underlying graph is fixed in most cases, the fact that edges can appear or disappear over time depending on the evolution of the opinions allows these models to describe the co-evolution of the opinion process and that of the social network. Some examples of this type of work covering different mechanisms in which the network changes are \cite{auletta2019consensus, friedkin2011formal, jia2015opinion, ye2017analysis}.

To end this section, we briefly mention that the study of stochastic recursions on random graphs and their connections to their local weak limits is an area of active research. A classical reference for stochastic recursions on trees and branching SFPEs is \cite{Aldo_Band_05}, and a rigorous connection between recursions on directed random graph models and the solutions to branching SFPEs can be found in \cite{Fra_Lin_Olv_22}. Related work for interacting diffusions on unimodular Galton-Watson trees can be found in \cite{Kavita_etal_2019, Kavita_etal_2020}. The notion of strong couplings and the many families of random graph models for which strong couplings exist can be found in \cite{Olvera_21}. A thorough discussion of complex networks and features such as the small-world phenomenon or the scale-free property that are prevalent in social networks can be found in \cite{Hofstad1,Hofstad2}. As we pointed out earlier, the work in this paper covers any directed random graph sequence having a locally finite limit, which includes most of the popular models used for replicating social networks. 

\section{Proofs of the main results}\label{sec:proofs}
In this section, we will prove all the results presented in the paper. For the reader's convenience, we have organized them in the same order in which they have appeared.

\subsection{Stationary Behavior}

We start with the proof of Theorem~\ref{thm:graphlim}, which establishes the convergence in distribution of the Markov chain $\{\mathbf{R}^{(k)}: k \geq 0\}$ on any fixed directed graph $G$. As mentioned earlier, the existence of a stationary distribution is a consequence of Theorem~1.1 in \cite{Diac_Freed_99}, however, we include a short proof here since it allows us to explicitly describe the connection between $\{\mathbf{R}^{(k)}: k \geq 0\}$ and the characterizations for a typical component provided by Theorem~\ref{thm:treelim}. The first step of the proof shifts the support of the recursion to make it nonnegative. The second step uses the fact that the external signals $\{ W_i^{(k)}: k\geq 0, i \in V\}$ are conditionally independent given the vertex marks $\{\mathbf{x}_i: i \in V\}$, with the $\{ W_i^{(k)}: k \geq 0\}$ conditionally i.i.d.~given $\mathbf{x}_i$, to create a time-reversed process that converges monotonically. The condition $d > 0$ is used to ensure that the limit does not depend on the initial vector $\mathbf{R}^{(0)}$, in other words, that the distribution of the limit corresponds to the unique stationary distribution of the process.


Let $G=(V,E;\mathscr{A})$ be a (potentially infinite) locally finite vertex-weighted directed graph with marks $\{\mathbf{x}_i: i \in V\}$.
Before going into the proof of the theorem, we need to introduce the supremum norm and its induced operator norm for functions of the vertices of $G$. For $f:V\to\mathbb{R}$ we define the norm
$$
\|f\|_\infty = \sup_{v\in V} f(v)
$$
and let $L^\infty(V) = \{f: \|f\|_\infty <\infty \}$. Next, we define an operator $\Delta$ related to our recursion \ref{eq:OpinionRec} that acts on $L^\infty(V)$ as follows 
\begin{equation}\label{eq:delta-def}
\Delta f(i) = \sum_{r=1}^{d_i^-} c(i,r) f(\ell(i,r)) + (1-c-d) f(i),
\end{equation}
where $\ell(i,r)$ is the vertex label of the $r$th inbound neighbor of $i$.
Finally, we define the operator norm
$$
\|\Delta\|_\infty = \sup_{f\in L^\infty(V)} \frac{\|\Delta f\|_\infty}{\|f\|_\infty}.
$$
Note that for all $i\in V$
$$
|\Delta f(i)| \leq \sum_{r=1}^{d_i^-} c(i,r) \|f\|_\infty + (1-c-d) \|f\|_\infty \leq (1-d) \|f\|_\infty
$$
which implies that
$\|\Delta\|_\infty \leq 1-d <1$. We are now ready to prove Theorem~\ref{thm:graphlim}. 

\begin{proof}[Proof of Theorem~\ref{thm:graphlim}.]
To start, note that the support of recursion \eqref{eq:OpinionRec} is $[-1,1]$. Therefore, by adding 1 to both sides of \eqref{eq:OpinionRec}, we obtain a new recursion given by 
\begin{equation}\label{eq:shifted_recursion}
\tilde{R}_i^{(k+1)} = \sum_{r = 1}^{d_i^-} c(i,r) \tilde{R}_{\ell(i,r)}^{(k)} +  \tilde{W}_i^{(k)} + (1-c-d) \tilde{R}_i^{(k)},
\end{equation}
with $\tilde{W}_i^{(k)} = W_i^{(k)}+d+c-\sum_{r=1}^{d_i^-} c(i,r)$. Since $\tilde{R}_i^{(k)} = R_i^{(k)}+1$, the new recursion has support $[0,2]$.

Note that the opinions $\mathbf{\tilde R}^{(k)}=\{R^{(k)}_i: i\in V\}$ and external signals $\mathbf{\tilde W}^{(k)}=\{W^{(k)}_i: i\in V\}$ are elements of $L^\infty(V)$. Moreover, using the $\Delta$ operator defined in \eqref{eq:delta-def} the shifted recursion \eqref{eq:shifted_recursion} can be written as 
$$
\mathbf{\tilde{R}}^{(k+1)} = \Delta \mathbf{\tilde{R}}^{(k)} + \mathbf{\tilde{W}}^{(k)}.
$$
Iterating this recursion we obtain that
$$
\mathbf{\tilde{R}}^{(k)} = \Delta^{k}\mathbf{\tilde{R}}^{(0)} + \sum_{r = 0}^{k-1} \Delta^r\mathbf{\tilde{W}}^{(k-r-1)}.
$$

Consider now the time-reversed process:
$$
\mathbf{\tilde{B}}^{(k)} := \sum_{r = 0}^{k-1} \Delta^r\mathbf{\tilde W}^{(r)},
$$
and note that since the external signals $\{\mathbf{\tilde{W}}^{(k)}: k\geq 0\}$ are conditionally i.i.d.~given the marks $\{ \mathbf{x}_i: i \in V\}$, we have that 
$$
\mathbf{\tilde{R}}^{(k)} \eqlaw \Delta^{k}\mathbf{\tilde{R}}^{(0)} + \mathbf{\tilde{B}}^{(k)}.
$$
Since all the summands in $\mathbf{\tilde{B}}^{(k)}$ are nonnegative, it follows that it is monotone increasing, and therefore  
$$
\mathbf{\tilde{B}}^{(k)} \nearrow   \sum_{r = 0}^\infty \Delta^r\mathbf{\tilde W}^{(r)} =: \mathbf{\tilde{B}} \qquad k \to \infty.
$$
Now use the triangle inequality and the observation that $\|\mathbf{\tilde W}^{(k)}\|_\infty\leq 2$ to obtain that
$$
\|\mathbf{\tilde{B}}^{(k)} - \mathbf{\tilde{B}} \|_\infty 
\leq \sum_{r = k}^\infty \| \Delta^r \mathbf{\tilde W}^{(k)}\|_\infty \leq 2 \sum_{r=k}^\infty \|\Delta\|_\infty^r \leq 2 \sum_{r=k}^\infty (1-d)^k \leq 2(1-d)^k (1/d) \to 0,
$$
as $k \to \infty$. Similarly, since $\|\mathbf{\tilde{R}}^{(0)}\|_\infty\leq 2$ we have
$$
\| \Delta^{k}\mathbf{\tilde{R}}^{(0)} \|_\infty \leq 2(1-d)^k \to 0,
$$
as $k \to \infty$. Combining the two upper bounds, it follows that
\begin{equation*}
\|\mathbf{\tilde{R}}^{(k)} - \mathbf{\tilde{B}}\|_\infty 
\leq \| \Delta^{k}\mathbf{\tilde{R}}^{(0)} \|_\infty + \|\mathbf{\tilde{B}}^{(k)} - \mathbf{\tilde{B}}\|_\infty
\leq (1-d)^{k}(2/d+2),
\end{equation*}
and so $ \|\mathbf{\tilde{R}}^{(k)} - \mathbf{\tilde{B}}\|_\infty \to 0$ as $k \to \infty$. This completes the proof of Theorem~\ref{thm:graphlim} because we have $\mathbf{\tilde{R}}^{(k)} \eqlaw \mathbf{R}^{(k)}+ \mathbf{1}$, so we can take $\mathbf{R} = \mathbf{\tilde{B}}-\mathbf{1}$, where $\mathbf{1}$ is the vector of ones. 
\end{proof}

\begin{remark} 
In the next sections we will take the graph to be random, in which case the convergence above is to be understood in a quenched sense. That is, if we let $\mathscr{G}_n = \sigma(G(V_n, E_n; \mathscr{A}_n))$ be the $\sigma$-algebra generated by the graph, and $\mathbf{P}_n(\,\cdot\,) = P(\,\cdot\,|\mathscr{G}_n)$ be the corresponding conditional probability measure, then Theorem~\ref{thm:graphlim} holds $\mathbf{P}_n$-almost surely.
\end{remark}

\subsection{Typical Behavior} \label{sec:proofThm2}

In this section we prove Theorem~\ref{thm:treelim}. From this point onwards, we assume that $\{ G(V_n, E_n; \mathscr{A}_n): n \geq 1\}$ is a sequence of random graphs for which a strong coupling with a locally finite marked, rooted, directed graph $\mathcal{G}(\boldsymbol{\mathcal{X}})$ exists. Recall that $I_n$ is a uniformly chosen vertex in $V_n$, which corresponds to the root $\emptyset$ of $\mathcal{G}(\boldsymbol{\mathcal{X}})$. 

Let $\mathscr{F}_n = \sigma( \mathscr{A}_n)$ denote the $\sigma$-algebra generated by the vertex attributes $\mathscr{A}_n$, which does not include the presence/absence of edges in the graph $G(V_n, E_n; \mathscr{A}_n)$, and let $\mathbb{P}_n(\,\cdot\, | \mathscr{F}_n)$ and $\mathbb{E}_n[ \,\cdot\, | \mathscr{F}_n]$ denote the corresponding conditional probability and expectation. 

Recall that $\{ \mathbf{R}^{(k)}: k \geq 0\}$ is the Markov chain describing the evolution of the opinions on the graph $G(V_n, E_n; \mathscr{A}_n)$. As such, the trajectories $\{ R_i^{(k)}: k \geq 0 \}$ for each of its vertices $i \in V_n$, tend to fluctuate over time. The key to the proof of Theorem~\ref{thm:graphlim} was the construction of a time-reversed process, $\{ \mathbf{\tilde B}^{(k)}: k \geq 0\}$, that had monotone trajectories and whose limit coincided with the stationary distribution of $\{ \mathbf{R}^{(k)} + \mathbf{1}: k \geq 0\}$. The proof of Theorem~\ref{thm:treelim} uses the original trajectories for the finite time part of the theorem, and the time-reversed versions for the statement regarding the stationary vector $\mathbf{R}$. 

To define the relevant processes, given $G(V_n, E_n; \mathscr{A}_n)$ we omit the dependence on $n$ and write $\bar \Delta$ for the operator defined in \eqref{eq:delta-def} that acts on $L^\infty(V_n)$ as:
$$
\bar \Delta f(i) = \sum_{r=1}^{D_i^-} C(i,r) f( L(i,r)) + (1-c-d) f(i),
$$
where $L(i,r)$ is the vertex label  of the $r$th inbound neighbor of $i$. Let $\mathbf{W}^{(k)} = \{ W_i^{(k)}: i \in V_n\}$ be the vector of external signals at time $k$, and note that from the proof of Theorem~\ref{thm:graphlim} we obtain that
$$
\mathbf{R}^{(k)} = \bar \Delta^k \mathbf{R}^{(0)} + \sum_{r=0}^{k-1} \bar \Delta^r \mathbf{W}^{(k-r-1)}.
$$
The time-reversed process we will use is:
$$
\mathbf{B}^{(k)} = \sum_{r=0}^k \bar \Delta^r \mathbf{W}^{(r)},
$$
which although it does not have monotone trajectories (since its support has not been shifted), converges almost surely, as $k \to \infty$, to the stationary vector $\mathbf{R}$. The $i$th coordinates of $\mathbf{R}^{(k)}, \mathbf{R}$, and $\mathbf{B}^{(k)}$ are denoted $R_i^{(k)}$, $R_i$, and $B_i^{(k)}$, respectively. 

We will start by proving a convergence result for the finite time processes $\mathbf{R}^{(k)}$ and $\mathbf{B}^{(k)}$. This is the key step needed in the second moment method that will later be used in the proof of Theorem~\ref{thm:treelim}, since it provides the asymptotic independence of two randomly chosen vertices.

\begin{theorem} \label{thm:LocalWeakConv}
Suppose the graph sequence $\{G(V_n, E_n; \mathscr{A}_n): n \geq 1\}$ admits a strong coupling with a marked, rooted, directed graph $\mathcal{G}(\boldsymbol{\mathcal{X}})$, and assume the conditions of Theorem~\ref{thm:graphlim} hold. Let $g,h: [-1,1] \to \mathbb{R}$ be two bounded and continuous functions, and let $I_n, J_n$ be two independent random variables uniformly chosen in $V_n$; let $\mathbf{R}^{(k)}$ and $\mathbf{B}^{(k)}$ be the vectors defined above. Then, for any $k \geq 1$ there exist i.i.d.~random sequences $\{(\mathcal{R}^{(t)}, \mathcal{B}^{(t)}): 0 \leq t \leq k\}$ and $\{ (\mathcal{\hat R}^{(t)}, \mathcal{\hat B}^{(t)}): 0 \leq t \leq k\}$, whose distribution does not depend on $\mathscr{F}_n$, such that
\begin{align*}
\sup_{0\leq t \leq k} \mathbb{E}_n\left[  \left|( g(R_{I_n}^{(t)}) - g(\mathcal{R}^{(t)}) ) (h(R_{J_n}^{(t)}) - h(\mathcal{\hat R}^{(t)}) ) \right|  \right] &\xrightarrow{P} 0 \\
\sup_{0 \leq t \leq k} \mathbb{E}_n\left[  \left|( g(B_{I_n}^{(t)}) - g(\mathcal{B}^{(t)}) ) (h(B_{J_n}^{(t)}) - h(\mathcal{\hat B}^{(t)}) ) \right|  \right]  &\xrightarrow{P} 0  \qquad n \to \infty.
\end{align*}
Moreover, $\mathcal{B}^{(k)}$ and $\mathcal{\hat B}^{(k)}$ converge almost surely as $k \to \infty$. 
\end{theorem}

\begin{proof}
To start, sample $I_n$ and $J_n$ uniformly from $V_n$ and independently of each other. For these two vertices $I_n$ and $J_n$ we will construct two independent rooted graphs $\mathcal{G}_\emptyset(\boldsymbol{\mathcal{X}})$ and $\mathcal{\hat G}_{\hat \emptyset}(\boldsymbol{\mathcal{X}})$, having the same distribution as $\mathcal{G}(\boldsymbol{\mathcal{X}})$, on which we can construct the random variables $\mathcal{R}^{(k)}, \mathcal{\hat R}^{(k)}, \mathcal{B}^{(k)}$ and $\mathcal{\hat B}^{(k)}$.

First, fix $k \geq 1$ and $\epsilon \in (0,1)$. Now let $G_{I_n}^{(k)}(\mathbf{X})$ and $G_{J_n}^{(k)}(\mathbf{X})$ denote the in-components of vertices $I_n$ and $J_n$ in $G(V_n, E_n; \mathscr{A}_n)$, respectively, and let $\mathcal{G}^{(k)}_\emptyset(\boldsymbol{\mathcal{X}})$ and $\mathcal{\hat G}_{\hat \emptyset}^{(k)}(\boldsymbol{\mathcal{X}})$ denote their strong couplings. Define the events $\mathcal{E}_{I_n}^{(k,\epsilon)}$ and $\mathcal{E}_{J_n}^{(k,\epsilon)}$ according to Definition~\ref{def:StrongCoupling}; recall that these correspond to the events that the inbound neighborhoods of depth $k$ of vertices $I_n$ and $J_n$ are isomorphic to their local weak limits, with their marks $\epsilon$ close to each other. Since we have a strong coupling for any fixed $k$, we can extend $\mathcal{G}_\emptyset^{(k)}(\boldsymbol{\mathcal{X}})$ in a consistent way so that $\mathcal{G}_\emptyset(\boldsymbol{\mathcal{X}}) := \lim_{m \to \infty} \mathcal{G}_\emptyset^{(m)}(\boldsymbol{\mathcal{X}}) \eqlaw \mathcal{G}(\boldsymbol{\mathcal{X}})$, and do the same with $\mathcal{\hat G}_{\hat \emptyset}^{(k)}(\boldsymbol{\mathcal{X}})$ in a way that the limiting graphs $\mathcal{G}_\emptyset (\boldsymbol{\mathcal{X}})$ and $\mathcal{\hat G}_{\hat \emptyset}(\boldsymbol{\mathcal{X}})$ are independent of each other. 

Next, let $V$ denote the set of vertices of $\mathcal{G}_\emptyset(\boldsymbol{\mathcal{X}})$ and let 
$$
\boldsymbol{\mathcal{X}}_i = (\boldsymbol{\mathcal{A}}_i, \mathcal{Q}_i, \mathcal{S}_i, \mathcal{N}_i, \mathcal{C}(i,1) 1(\mathcal{N}_i \geq 1), \mathcal{C}(i,2) 1(\mathcal{N}_i \geq 2), \dots )
$$
denote the mark of vertex $i \in V$. Analogously to \eqref{eq:delta-def} define the operator $\Delta$ acting on $L^\infty(V)$ as:
$$
\Delta f(i) = \sum_{r=1}^{\mathcal{N}_i} \mathcal{C}(i,r) f( \mathcal{L}(i,r)) + (1-c-d) f(i), \qquad i \in V,
$$
where $\mathcal{L}(i,r)$ is the $r$th inbound neighbor of vertex $i \in V$. Let $\hat V$ denote the vertices of $\mathcal{\hat G}_{\hat \emptyset}(\boldsymbol{\mathcal{X}})$, $\{ \boldsymbol{\mathcal{\hat X}}_i\}$ the vertex marks, and $\hat \Delta$ the corresponding operator on $L^\infty(\hat V)$. 

We now need to construct suitable sequences of external signals that are optimally coupled. Let $V_k, \hat V_k, V_{k,I_n}$ and $V_{k,J_n}$ denote the set of vertices of $\mathcal{G}^{(k)}_\emptyset, \mathcal{\hat G}_{\hat \emptyset}^{(k)}, G_{I_n}^{(k)}$, and $G_{J_n}^{(k)}$, respectively. Let $\theta: V_k \to V_{k,I_n}$ and $\hat \theta: \hat V_k \to V_{k,J_n}$ denote the bijections determining the strong couplings on the event. $\mathcal{E}_{I_n}^{(k,\epsilon)} \cap \mathcal{E}_{J_n}^{(k,\epsilon)}$. For any vertex $i \in V_k$ and any time $r \geq 0$, set the external signals at time $r$ to be:
$$
\left(\mathcal{W}_i^{(r)}, W_{\theta(i)}^{(r)} \right) = \left( H_{\boldsymbol{\mathcal{X}}_i}^{-1}(U_i^{(r)}), H_{\mathbf{X}_{\theta(i)}}^{-1}(U_i^{(r)}) \right),
$$
where $H_{\mathbf{x}}(w) = P( W_i^{(1)} \leq w \mid \mathbf{X}_i = \mathbf{x})$, $H^{-1}(u) = \inf\{ w \in \mathbb{R}: H(w) > u\}$, and the $\{U_i^{(r)}\}$ are i.i.d.~uniform $[0,1]$ random variables. This makes $(\mathcal{W}_i^{(r)}, W_{\theta (i)}^{(r)})$ an optimal coupling under any Wasserstein metric. Repeat this for all vertices $j \in \hat V_k$ and all $r \geq 0$ to obtain:
$$
\left(\mathcal{\hat W}_j^{(r)}, W_{\hat \theta(j)}^{(r)} \right) = \left( H_{\boldsymbol{\mathcal{\hat X}}_j}^{-1}(U_j^{(r)}), H_{\mathbf{X}_{\hat \theta(j)}}^{-1}(U_j^{(r)}) \right).
$$
For each vertex $i \in V \setminus V_k$, sample $\{\mathcal{W}_i^{(r)}: r \geq 0\}$ to be i.i.d.~with distribution $\nu(\boldsymbol{\mathcal{X}}_i)$, independent of everything else, and repeat for all vertices $j \in \hat V \setminus \hat V_k$. Finally, for each vertex $i \in V_n \setminus (V_{k,I_n} \cup V_{k,J_n})$ sample $\{ W_i^{(r)}: r\geq 0\}$ i.i.d.~with distribution $\nu(\mathbf{X}_i)$, independently of everything else. Now let for each $r \geq 0$, $\mathbf{W}^{(r)} = \{ W_i^{(r)}: i \in V_n\}$, $\boldsymbol{\mathcal{W}}^{(r)} = \{ \mathcal{W}_i^{(r)}: i \in V\}$ and $\boldsymbol{\mathcal{\hat W}}^{(r)} = \{ \mathcal{\hat W}_i^{(r)}: i \in V\}$; note that $\{ (\mathbf{W}^{(r)}, \boldsymbol{\mathcal{W}}^{(r)},\boldsymbol{\mathcal{\hat W}}^{(r)}  ): r \geq 0\}$ are i.i.d.~by construction. For the last step of the construction, define the vectors $\boldsymbol{\mathcal{R}}^{(0)} = \{ \mathcal{R}^{(0)}_i: i \in V\}$ and $\boldsymbol{\mathcal{\hat R}}^{(0)} = \{ \mathcal{\hat R}_i: i \in \hat V\}$ so that $\mathcal{R}^{(0)}_i = R_{\theta(i)}^{(0)}$ for $i \in V_k$ and $\mathcal{\hat R}^{(0)}_j = R_{\hat\theta(j)}^{(0)}$ for $j \in \hat V_k$, with the initial values on all other vertices chosen arbitrarily. 

We note that the construction given above holds on the event $\mathcal{E}_{I_n}^{(k,\epsilon)} \cap \mathcal{E}_{J_n}^{(k,\epsilon)}$, which occurs with high probability. On the complementary event, we can define all the relevant graphs and vectors to be independent of each other. 

Now define for each $0 \leq t \leq k$:
\begin{align*}
\boldsymbol{\mathcal{R}}^{(t)} &= \Delta^t \boldsymbol{\mathcal{R}}^{(0)} + \sum_{r=0}^{t-1} \Delta^r \boldsymbol{\mathcal{W}}^{(t-1-r)}, \quad \boldsymbol{\mathcal{\hat R}}^{(t)} = \hat\Delta^t \boldsymbol{\mathcal{\hat R}}^{(0)} + \sum_{r=0}^{t-1} \hat\Delta^r \boldsymbol{\mathcal{\hat W}}^{(t-1-r)}, \\
\boldsymbol{\mathcal{B}}^{(t)} &= \sum_{r=0}^{t} \Delta^r \boldsymbol{\mathcal{W}}^{(r)}, \hspace{11mm} \text{and} \hspace{11mm}  \boldsymbol{\mathcal{\hat B}}^{(t)} =  \sum_{r=0}^{t} \hat\Delta^r \boldsymbol{\mathcal{\hat W}}^{(r)}.
\end{align*}
The random variables in the statement of the theorem are:
$$
\mathcal{R}^{(t)} := \mathcal{R}^{(t)}_\emptyset, \quad \mathcal{\hat R}^{(t)} := \mathcal{\hat R}^{(t)}_{\hat \emptyset}, \quad \mathcal{B}^{(t)} := \mathcal{B}^{(t)}_\emptyset, \quad \text{and} \quad \mathcal{\hat B}^{(t)} := \mathcal{\hat B}^{(t)}_{\hat\emptyset},
$$
where $\boldsymbol{\mathcal{R}}^{(t)} = \{ \mathcal{R}_i^{(t)}: i \in V\}$, $\boldsymbol{\mathcal{\hat R}}^{(t)} = \{ \mathcal{\hat R}_i^{(t)}: i \in \hat V\}$, $\boldsymbol{\mathcal{B}}^{(t)} = \{ \mathcal{B}_i^{(t)}: i \in V\}$, and $\boldsymbol{\mathcal{\hat B}}^{(t)} = \{ \mathcal{\hat B}_i^{(t)}: i \in \hat V\}$ for each $0 \leq t \leq k$. 

We now prove the statement of the theorem, for which we start by noting that $g$ and $h$ are uniformly continuous on $[-1,1]$. Let $\omega_g(\delta) = \sup\{ |g(x) - g(y)|: |x-y| \leq \delta\}$ and $\omega_h(\delta) = \sup\{|h(x) - h(y)|: |x-y| \leq \delta\}$. Now note that for any $\delta \in (0,1)$:
\begin{align*}
&\sup_{0\leq t \leq k} \mathbb{E}_n\left[  \left|( g(R_{I_n}^{(t)}) - g(\mathcal{R}^{(t)}) ) (h(R_{J_n}^{(t)}) - h(\mathcal{\hat R}^{(t)}) ) \right|  \right] \\
&\leq \| h\|_\infty \| g \|_\infty \mathbb{P}_n\left( \left( \mathcal{E}_{I_n}^{(k,\epsilon)} \cap \mathcal{E}_{J_n}^{(k,\epsilon)} \right)^c \right) + \| h\|_\infty \| g \|_\infty \max_{0\leq t \leq k} \mathbb{P}_n\left( \mathcal{E}_{I_n}^{(k,\epsilon)} , \, \left| R_{I_n}^{(t)} - \mathcal{R}^{(t)} \right| > \delta  \right) \\
&\hspace{5mm} + \| h\|_\infty \| g \|_\infty \max_{0\leq t \leq k} \mathbb{P}_n\left( \mathcal{E}_{J_n}^{(k,\epsilon)} , \, \left| R_{J_n}^{(t)} - \mathcal{\hat R}^{(t)} \right| > \delta  \right) + \omega_g(\delta) \omega_f(\delta) .
\end{align*}
Note that the first conditional probability converges to zero in probability, as $n\to\infty$, by Definition~\ref{def:StrongCoupling}, while the second and third conditional probabilities are equal to each other since $I_n \eqlaw J_n$. 

To analyze the remaining probability, let $\mathcal{D}_r$ denote the set of vertices in $\mathcal{G}_\emptyset^{(k)}$ that have a directed path of length $r$ to $\emptyset$. Let $\mathscr{H}_k = \sigma\left( G_{I_n}^{(k)}(\mathbf{X}), \mathcal{G}_\emptyset^{(k)}(\boldsymbol{\mathcal{X}}) \right)$. Now define on the event $\mathcal{E}_{I_n}^{(k,\epsilon)}$, and for vertices $i \in \mathcal{G}_\emptyset$ and any $0\leq r \leq t \leq k$:
$$
J_i^{(t,r)} = E\left[ \left. \left| R_{\theta(i)}^{(r)} - \mathcal{R}_i^{(r)} \right| \right| \mathscr{H}_k\right] 1( i \in \mathcal{D}_{t-r} ),
$$
and let $\mathbf{J}^{(t,r)}  = \{ J_i^{(t,r)}: i \in V\}$. Now note that
\begin{align*}
J_\emptyset^{(t,t)} &\leq E\left[  \left| \sum_{j=1}^{\mathcal{N}_\emptyset} \mathcal{C}(\emptyset, j) \left( \mathcal{R}_{\mathcal{L}(\emptyset,j)}^{(t-1)} - R_{L(I_n,j)}^{(t-1)}  \right)  + (1-c-d) \left( R_{I_n}^{(t-1)}  -  \mathcal{R}_\emptyset^{(t-1)} \right)  \right| \right. \\
&\hspace{5mm} + \left. \left. \left|   \sum_{j=1}^{\mathcal{N}_\emptyset} \left( C(I_n,j) -  \mathcal{C}(\emptyset,j) \right)  R_{L(I_n,j)}^{(t-1)}  \right| + \left|  W_{I_n}^{(t-1)} - \mathcal{W}_\emptyset^{(t-1)} \right| \right| \mathscr{H}_k \right]  \\
&\leq \left( \Delta \mathbf{J}^{(t,t-1)} \right)_\emptyset + \rho(\mathbf{X}_{I_n}, \boldsymbol{\mathcal{X}}_\emptyset) E\left[ \left.  \max_{1 \leq j \leq \mathcal{N}_\emptyset} \left| R_{L(I_n, j)}^{(t-1)} \right| \right| \mathscr{H}_k \right] +K  \rho(\mathbf{X}_{I_n},\boldsymbol{\mathcal{X}}_\emptyset)  \\
&\leq  \left( \Delta \mathbf{J}^{(t,t-1)} \right)_\emptyset + (2+K) \epsilon   .
\end{align*}
It follows that $\| \mathbf{J}^{(t,t)} \|_\infty \leq \| \Delta \mathbf{J}^{(t,t-1)} \|_\infty + (2+K)\epsilon \leq || \Delta \|_\infty \| \mathbf{J}^{(t,t-1)} \|_\infty + (2+K)\epsilon$. Note that the same steps give for any $i \in \mathcal{D}_1$ that
$$
J_i^{(t,t-1)} \leq \left( \Delta \mathbf{J}^{(t,t-2)} \right)_i + (2+K)\epsilon, \qquad \text{and} \qquad \| \mathbf{J}^{(t,t-1)} \|_\infty \leq \| \Delta\|_\infty  \| \mathbf{J}^{(t,t-2)} \|_\infty + (2+K)\epsilon.
$$
Iterating for $t-2$ more steps and using that $\| \Delta\|_\infty \leq 1-d$, gives
$$
\| \mathbf{J}^{(t,t)} \|_\infty \leq \sum_{r=0}^{t-1} \| \Delta \|_\infty^r (2+K)\epsilon + \| \Delta \|_\infty^t \| \mathbf{J}^{(t,0)} \|_\infty \leq (2+K)\epsilon \sum_{r=0}^{t-1} (1-d)^r.
$$

We now obtain by conditioning on $\mathscr{H}_k$ that
\begin{align*}
\max_{0\leq t \leq k} \mathbb{P}_n\left( \mathcal{E}_{I_n}^{(k,\epsilon)} , \, \left| R_{I_n}^{(t)} - \mathcal{R}^{(t)} \right| > \delta  \right)  &\leq \max_{0\leq t \leq k} \mathbb{E}_n\left[ 1\left( \mathcal{E}_{I_n}^{(k,\epsilon)}   \right)  \delta^{-1} \| \mathbf{J}^{(t,t)} \|_\infty \right] \\
&\leq \delta^{-1} (2+K)\epsilon \sum_{r=0}^{k-1} (1-d)^r.
\end{align*}
We conclude that
\begin{align*}
&\sup_{0\leq t \leq k} \mathbb{E}_n\left[  \left|( g(R_{I_n}^{(t)}) - g(\mathcal{R}^{(t)}) ) (h(R_{J_n}^{(t)}) - h(\mathcal{\hat R}^{(t)}) ) \right|  \right] \\
&\leq \| h\|_\infty \| g \|_\infty \mathbb{P}_n\left( \left( \mathcal{E}_{I_n}^{(k,\epsilon)} \cap \mathcal{E}_{J_n}^{(k,\epsilon)} \right)^c \right) + 2 \| h\|_\infty \| g \|_\infty  \delta^{-1} (2+K)\epsilon \sum_{r=0}^{k-1} (1-d)^r + \omega_g(\delta) \omega_f(\delta) \\
&\xrightarrow{P} 2 \| h\|_\infty \| g \|_\infty  \delta^{-1} (2+K)\epsilon \sum_{r=0}^{k-1} (1-d)^r + \omega_g(\delta) \omega_f(\delta)
\end{align*}
as $n \to \infty$. Taking $\epsilon \downarrow 0$ followed by $\delta \downarrow 0$ gives the first statement of the theorem. The one involving the time reversed processes follows exactly the same steps, so the proof is complete. 
\end{proof}

We are now ready to prove Theorem~\ref{thm:treelim}. 

\begin{proof}[Proof of Theorem~\ref{thm:treelim}.]
We start with the statements in (1), involving the original recursions. The statement about the expectations was proven within the proof of Theorem~\ref{thm:LocalWeakConv} (it can also be directly derived from Theorem~\ref{thm:LocalWeakConv} by choosing $g(x) = h(x) = x$ and using the monotonicity of the $L_p$ norm). For the second statement let $\mu_k = E[f(\mathcal{R}^{(k)})]$ and note that by Theorem~\ref{thm:LocalWeakConv}, there exist i.i.d.~random variables $\mathcal{R}^{(k)}, \mathcal{\hat R}^{(k)}$ such that
\begin{align*}
E\left[ \left( \frac{1}{n} \sum_{i=1}^n f(R_i^{(k)}) - \mu_k \right)^2 \right] &= \frac{1}{n^2} \sum_{i=1}^n \sum_{j =1}^n E\left[ \left( f(R_i^{(k)}) - \mu_k \right) \left( f(R_j^{(k)}) - \mu_k \right) \right] \\
&= E\left[  \left( f(R_{I_n}^{(k)}) - \mu_k \right) \left( f(R_{J_n}^{(k)}) - \mu_k \right) \right] \\
&= E\left[  \left( f(R_{I_n}^{(k)}) - f(\mathcal{R}^{(k)})  \right) \left( f(R_{J_n}^{(k)}) - f(\mathcal{\hat R}^{(k)})  \right) \right]  \\
&\hspace{5mm} + E\left[  \left( f(\mathcal{R}^{(k)}) -  \mu_k \right) \left( f(R_{J_n}^{(k)}) - f(\mathcal{\hat R}^{(k)})  \right) \right] \\
&\hspace{5mm} +  E\left[  \left( f(R_{I_n}^{(k)}) - f(\mathcal{R}^{(k)})  \right) \left( f(\mathcal{\hat R}^{(k)}) - \mu_k \right) \right] \\
&\hspace{5mm} + E\left[  \left( f(\mathcal{R}^{(k)}) -  \mu_k \right) \left(  f(\mathcal{\hat R}^{(k)}) - \mu_k \right) \right] \\
&\leq E\left[ \mathbb{E}_n\left[ \left| \left( f(R_{I_n}^{(k)}) - f(\mathcal{R}^{(k)})  \right) \left( f(R_{J_n}^{(k)}) - f(\mathcal{\hat R}^{(k)})  \right) \right|  \right] \right]   \\
&\hspace{5mm} + 2 \| f\|_\infty E\left[ \mathbb{E}_n\left[ \left| f(R_{I_n}^{(k)}) - f(\mathcal{R}^{(k)})  \right| \right] \right] .
\end{align*}
Now use Theorem~\ref{thm:LocalWeakConv} and the bounded convergence theorem to obtain that
$$
\lim_{n \to \infty} E\left[ \left( \frac{1}{n} \sum_{i=1}^n f(R_i^{(k)}) - \mu_k \right)^2 \right] = 0,
$$
from where the convergence in probability follows. 

For the statements in part (2) we use the time-reversed processes and take $\mathbf{R} = \sum_{r=0}^\infty \bar \Delta^r \mathbf{W}^{(r)}$ and $\mathcal{R}^* = \lim_{k\to \infty} \mathcal{B}^{(k)}$. For the expectation result, note that for any $k \geq 1$, 
\begin{align*}
\mathbb{E}_n\left[ \left| R_{I_n} - \mathcal{R}^* \right| \right] &\leq \mathbb{E}_n\left[ \left| R_{I_n} - B_{I_n}^{(k)}  \right| \right] + \mathbb{E}_n\left[ \left| B_{I_n}^{(k)}  - \mathcal{B}^{(k)} \right| \right] + E\left[ \left| \mathcal{B}^{(k)} - \mathcal{R}^* \right| \right] .\end{align*}
Now note that
\begin{align*}
\mathbb{E}_n\left[ \left|R_{I_n} - B_{I_n}^{(k)}  \right| \right]  &\leq \mathbb{E}_n\left[ \| \mathbf{R} - \mathbf{B}^{(k)} \|_\infty \right] \leq \sum_{r=k+1}^\infty \mathbb{E}_n\left[ \| \bar \Delta^r \mathbf{W}^{(r)} \|_\infty \right] \\
&\leq \sum_{r=k+1}^\infty \mathbb{E}_n\left[ \| \bar \Delta \|_\infty^r \| \mathbf{W}^{(r)} \|_\infty \right]  \leq 2 \sum_{r=k+1}^\infty (1-d)^r.
\end{align*}
Therefore,
$$
\lim_{n \to \infty} \mathbb{E}_n\left[ \left| R_{I_n} - \mathcal{R}^* \right| \right] \leq 2 \sum_{r=k+1}^\infty (1-d)^r + E\left[ \left| \mathcal{B}^{(k)} - \mathcal{R}^* \right| \right] ,
$$
and the result follows from taking $k \to \infty$. To obtain the convergence in probability result, let $\mu = E[f(\mathcal{R}^*)]$ and $\lambda_k = E[f(\mathcal{B}^{(k)})]$, and note that for any $k \geq 1$, 
\begin{align*}
\left| \frac{1}{n} \sum_{i=1}^n f(R_i) -\mu \right| &\leq \left| \frac{1}{n} \sum_{i=1}^n (f(R_i) - f(B_i^{(k)}) )\right| + \left| \frac{1}{n} \sum_{i=1}^n f(B_i^{(k)}) - \lambda_k \right| + |\lambda_k - \mu|.
\end{align*}
The second term on the right converges to zero in probability as $n \to \infty$ by the same arguments used in part (1). And the first term satisfies for any $\delta > 0$ and $\omega_f(\delta) = \sup\{ |f(x) - f(y)|: |x-y|\leq \delta\}$,
\begin{align*}
\left| \frac{1}{n} \sum_{i=1}^n (f(R_i) - f(B_i^{(k)}) )\right| &\leq \frac{1}{n} \sum_{i=1}^n | f(R_i) - f(B_i^{(k)}) | 1( |R_i-B_i^{(k)}| > \delta) +  \omega_f(\delta) \\
&\leq \frac{ \|f \|_\infty}{\delta n} \sum_{i=1}^n |R_i - B_i^{(k)} | + \omega_f(\delta) \leq \frac{ \| f \|_\infty \| \mathbf{R} - \mathbf{B}^{(k)} \|_\infty}{\delta} +  \omega_f(\delta) \\
&\leq \frac{2\| f\|_\infty}{\delta} \sum_{r=k+1}^\infty (1-d)^r +  \omega_f(\delta). 
\end{align*}
We conclude that
$$
\lim_{n \to \infty} \left| \frac{1}{n} \sum_{i=1}^n (f(R_i) - f(B_i^{(k)}) )\right| \leq  \frac{2\| f\|_\infty}{\delta} \sum_{r=k+1}^\infty (1-d)^r +  \omega_f(\delta) + |\lambda_k - \mu |,
$$
and taking $k \to \infty$ followed by $\delta \downarrow 0$ proves the convergence in probability. That $\mathcal{R}^{(k)} \Rightarrow \mathcal{R}^*$ as $k \to \infty$ follows from their explicit constructions given in the proof of Theorem~\ref{thm:treelim}. This completes the proof. 
\end{proof} 

\subsection{Explicit Characterizations}\label{sec:explicit_proofs}

In this section we present the proofs of Propositions \ref{prop:explicit}, \ref{prop:consensus}, \ref{prop:polarization} and \ref{prop:memory}. We begin by providing an explicit characterization of $\mathcal{R}^*$ when the local limit $\mathcal{G}(\boldsymbol{\mathcal{X}})$ is a tree.

\begin{proof}[Proof of Proposition \ref{prop:explicit}.]
Let $\boldsymbol{\mathcal{R}}$ be the vector of stationary opinions in the limit tree $\mathcal{T}(\boldsymbol{\mathcal{X}})$ (not necessarily a Galton-Watson tree). From the proof of Theorem \ref{thm:graphlim} we know that $\boldsymbol{\mathcal{R}} \eqlaw \sum_{s=0}^\infty \Delta^s \boldsymbol{\mathcal{W}}^{(s)}$ where $\Delta$ is defined in \eqref{eq:delta-def}. Therefore, it is enough to show that $\Delta^s$ acts on the tree according to 
$$
\Delta^s f(\mathbf{i}) = \sum_{l=0}^s \sum_{|\mathbf{j}|=l} \frac{\Pi_{(\mathbf{i}, \mathbf{j})} }{\Pi_{\mathbf{i}}} a_{l,s} f((\mathbf{i}, \mathbf{j})),
$$
for all $\mathbf{i}\in \mathcal{T}(\boldsymbol{\mathcal{X}})$. We prove this by induction. 

Without loss of generality, let $\mathbf{i} \in \mathcal{T}(\boldsymbol{\mathcal{X}})$ with $|\mathbf{i}| = r$, for some fix $r \geq 0$. When $s = 0$,
\begin{align*}
    \Delta^0 f(\mathbf{i})  = \sum_{|\mathbf{j}|=0} \frac{\Pi_{(\mathbf{i}, \mathbf{j})} }{\Pi_{\mathbf{i}}} a_{0,0} f((\mathbf{i}, \mathbf{j})) = a_{0,0}f(\mathbf{i}) = f(\mathbf{i}).
\end{align*}
Next, suppose that the expression holds for $\Delta^{k}$. Then, applying $\Delta$ once we obtain that
\begin{align*}
\Delta^{(k+1)}f(\mathbf{i}) & = \sum_{m = 1}^{\mathcal{N}_\mathbf{i}} \mathcal{C}_{(\mathbf{i}, m)} \Delta^k f((\mathbf{i}, m)) + (1-c-d) \Delta^k f(\mathbf{i}) \\
& = \sum_{m = 1}^{\mathcal{N}_\mathbf{i}} \mathcal{C}_{(\mathbf{i}, m)} \sum_{l = 0}^{k}\sum_{|\mathbf{j}|=l} \frac{\Pi_{(\mathbf{i}, m, \mathbf{j})} }{\Pi_{(\mathbf{i}, m)}}a_{l,k} f((\mathbf{i}, m, \mathbf{j}))
+ (1-c-d) \sum_{l = 0}^{k}\sum_{ |\mathbf{j}|=l} \frac{\Pi_{(\mathbf{i}, \mathbf{j})} }{\Pi_{\mathbf{i}}}a_{l,k} f((\mathbf{i}, \mathbf{j})).
\end{align*}
Recall that $\Pi_{(\mathbf{i}, m)} = \Pi_{\mathbf{i}}\mathcal{C}_{(\mathbf{i},m)}$, let $l' = l+1$, and rearrange the terms in the first sum to obtain
\begin{align*}
\Delta^{(k+1)}f(\mathbf{i}) 
& =  \sum_{l' = 1}^{k+1}\sum_{{ |\mathbf{j}| = l'}} \frac{\Pi_{(\mathbf{i}, \mathbf{j})} }{\Pi_{\mathbf{i}}}a_{l'-1,k} f((\mathbf{i}, \mathbf{j}))
+ (1-c-d) \sum_{l = 0}^{k}\sum_{{ |\mathbf{j}| = l}} \frac{\Pi_{(\mathbf{i}, \mathbf{j})} }{\Pi_{\mathbf{i}}}a_{l,k} f((\mathbf{i}, \mathbf{j})).
\end{align*}
Combining the common terms yields
\begin{align*}
\Delta^{(k+1)}f(\mathbf{i}) 
& =  \sum_{|\mathbf{j}| = k+1} \frac{\Pi_{(\mathbf{i}, \mathbf{j})} }{\Pi_{\mathbf{i}}}a_{k,k} f((\mathbf{i}, \mathbf{j})) + \sum_{l = 1}^{k} \sum_{ |\mathbf{j}|=l} \frac{\Pi_{(\mathbf{i}, \mathbf{j})} }{\Pi_{\mathbf{i}}}(a_{l-1,k}+(1-c-d)a_{l,k})  f((\mathbf{i}, \mathbf{j})) \\
& \hspace{5mm} +  (1-c-d)a_{0,k} f(\mathbf{i}).
\end{align*}
Now note that 
$$
a_{l-1, k} +(1-c-d)a_{l,k} = \left(\binom{k}{l-1} + \binom{k}{l}\right) (1-c-d)^{k+1-l} = a_{l,k+1},
$$
$$
a_{k,k} = a_{k+1, k+1} = 1, \qquad \text{ and } \qquad  (1-c-d)a_{0,k} = a_{0,k+1},
$$ 
from where it follows that
\begin{align*}
    \Delta^{(k+1)}f(\mathbf{i}) 
    & =  \sum_{l = 0}^{k+1} \sum_{ |\mathbf{j}| = l} \frac{\Pi_{(\mathbf{i}, \mathbf{j})} }{\Pi_{\mathbf{i}}}a_{l,k+1}  f((\mathbf{i}, \mathbf{j})).
\end{align*}
This completes the proof. 
\end{proof}

We now present the proofs of Proposition \ref{prop:consensus} and Proposition \ref{prop:polarization} on the expected value and variance of the typical stationary opinion $\mathcal{R}^*$. 

\begin{proof}[Proof of Propositions \ref{prop:consensus} and \ref{prop:polarization}.]
Recall that, as shown in Remark \ref{rem:SFPE}, under the assumption that the sequence of random graphs converges to a delayed marked Galton-Watson tree, the limiting opinion $\mathcal{R}^*$ for the recursion with no memory (c+d = 1) and $P(\mathcal{N}_\emptyset > 0) = 1$ can be written as:
$$
\mathcal{R}^*  = \sum_{i = 1}^{\mathcal{N}_\emptyset}\mathcal{C}_i\mathcal{R}_i + d\mathcal{Z}_\emptyset,
$$
where the $\{\mathcal{R}_i\}$ are independent of $\boldsymbol{\mathcal{X}}_\emptyset$ and i.i.d. copies of the solution to the stochastic fixed point equation:
$$
\mathcal{R} \eqlaw \sum_{i = 1}^{\mathcal{N}_1}\mathcal{C}_{(1,i)}\mathcal{R}_i + d\mathcal{Z}_1,
$$
with the $\{\mathcal{R}_i\}$ i.i.d. copies of $\mathcal{R}$, independent of $\boldsymbol{\mathcal{X}}_1$. The above recursions can be used directly to compute the mean of $\mathcal{R}^*$ and its higher moments. Recall that, due to the size-bias encountered in the exploration of the graph, the limiting tree is a delayed Galton-Watson tree, meaning its root $\emptyset$ may have a vertex mark $\boldsymbol{\mathcal{X}}_\emptyset$ that has a different distribution from all other vertex marks $\{ \boldsymbol{\mathcal{X}}_\mathbf{i}: \mathbf{i} \neq \emptyset\}$. This, in turn, makes the stationary opinion of the root, $\mathcal{R}^*$, also different from that of all other nodes, $\mathcal{R}$.  The conditional expectation and conditional variance of $\mathcal{R}^*$ given $\boldsymbol{\mathcal{X}}_\emptyset$ take the form:
\begin{align}
E[\mathcal{R}^*|\boldsymbol{\mathcal{X}}_\emptyset] & = E\left[\left.\sum_{i = 1}^{\mathcal{N}_\emptyset}\mathcal{C}_i\mathcal{R}_i\right|\boldsymbol{\mathcal{X}}_\emptyset\right] + d E[\mathcal{Z}_\emptyset|\boldsymbol{\mathcal{X}}_\emptyset]   = cE[\mathcal{R}] + d E[\mathcal{Z}_\emptyset|\boldsymbol{\mathcal{X}}_\emptyset], \label{eq:Ezero}\\
\var(\mathcal{R}^*|\boldsymbol{\mathcal{X}}_\emptyset) &= \var\left(\left.\sum_{i = 1}^{\mathcal{N}_\emptyset}\mathcal{C}_i\mathcal{R}_i + d\mathcal{Z}_\emptyset\right|\boldsymbol{\mathcal{X}}_\emptyset\right) = \sum_{i = 1}^{\mathcal{N}_\emptyset}\mathcal{C}_i^2\var(\mathcal{R}) + d^2\var(\mathcal{Z}_\emptyset |\boldsymbol{\mathcal{X}}_\emptyset). \label{eq:Vzero}
\end{align}
Similarly, the conditional expectation and variance of $\mathcal{R}$ are given by
\begin{align*}
E[\mathcal{R}|\boldsymbol{\mathcal{X}}_1] = cE[\mathcal{R}] + d E[\mathcal{Z}_1|\boldsymbol{\mathcal{X}}_1]
\quad\text{and}\quad
\var(\mathcal{R}|\boldsymbol{\mathcal{X}}_1) = \sum_{i = 1}^{\mathcal{N}_1}\mathcal{C}{(1,i)}^2\var(\mathcal{R}) + d^2\var(\mathcal{Z}_1 |\boldsymbol{\mathcal{X}}_1).
\end{align*}
Taking expected values over $\boldsymbol{\mathcal{X}}_1$ we get that $E[\mathcal{R}] = cE[\mathcal{R}] + d E[\mathcal{Z}_1]$ and
\begin{align*}
\var(\mathcal{R}) = 
E[\var(\mathcal{R}|\boldsymbol{\mathcal{X}}_1)] + \var(E[\mathcal{R}|\boldsymbol{\mathcal{X}}_1])
= \rho_2\var(\mathcal{R}) + d^2 \var(\mathcal{Z}_1),
\end{align*}
where $\rho_2 = E\left[\sum_{i = 1}^{\mathcal{N}_1}\mathcal{C}_{(1,i)}^2\right]$. Using $c+d = 1$, we conclude that
\begin{align}\label{eq:EV1}
E\left[\mathcal{R}\right]  = E[\mathcal{Z}_1],
\quad\text{and}\quad
\var(\mathcal{R})  = \frac{d^2}{1-\rho_2}\var(\mathcal{Z}_1).
\end{align}
Combining \eqref{eq:EV1} with \eqref{eq:Ezero} and \eqref{eq:Vzero}, we obtain
\begin{align*}
E[\mathcal{R}^*|\boldsymbol{\mathcal{X}}_\emptyset] & =  cE\left[\mathcal{Z}_1\right] + dE[\mathcal{Z}_\emptyset|\boldsymbol{\mathcal{X}}_\emptyset], \qquad \text{and} \\    \var(\mathcal{R}^*|\boldsymbol{\mathcal{X}}_\emptyset) & = \frac{d^2}{1-\rho_2}\sum_{i = 1}^{\mathcal{N}_\emptyset}\mathcal{C}_i^2\var(\mathcal{Z}_1) + d^2\var\left( \mathcal{Z}_\emptyset\middle|\boldsymbol{\mathcal{X}}_\emptyset\right),
\end{align*}
as claimed in Proposition \ref{prop:polarization}.
Moreover, averaging over $\boldsymbol{\mathcal{X}}_\emptyset$, we compute the expectation of $\mathcal{R}^*$ to be
$$
E[\mathcal{R}^*] =  cE\left[\mathcal{Z}_1\right] + dE[\mathcal{Z}_\emptyset].
$$
And using the law of total variance for $\mathcal{R}^*$ we obtain that
\begin{align*}
\var(\mathcal{R}^*) & = E[\var(\mathcal{R}^*|\boldsymbol{\mathcal{X}}_\emptyset)] + \var(E[\mathcal{R}^*|\boldsymbol{\mathcal{X}}_\emptyset]) \\
& = \frac{d^2\rho_2^*}{1-\rho_2}\var(\mathcal{Z}_1) + d^2E[\var(\mathcal{Z}_\emptyset|\boldsymbol{\mathcal{X}}_\emptyset)] + 
d^2\var(E[\mathcal{Z}_\emptyset|\boldsymbol{\mathcal{X}}_\emptyset]) \\
& = \frac{d^2\rho_2^*}{1-\rho_2}\var(\mathcal{Z}_1) + d^2\var(\mathcal{Z}_\emptyset),
\end{align*}
where $\rho_2^* = E\left[\sum_{i = 1}^{\mathcal{N}_\emptyset}\mathcal{C}_i^2\right]$ and $\rho_2 = E\left[\sum_{i = 1}^{\mathcal{N}_1}\mathcal{C}_{(1,i)}^2\right]$. The expressions for the expected value and variance of $\mathcal{R}^*$ yield the result in Proposition \ref{prop:consensus}. 
\end{proof}

We finish the section proving Proposition \ref{prop:memory} which compares the variance of the limiting distribution for the general case $c+d < 1$, with that of the no memory case $c+d = 1$. The proportional influence of the media signals and of the neighbors' opinions is preserved by renormalizing according to $c+d$ in the general recursion.

\begin{proof}[Proof of Proposition \ref{prop:memory}.]
When the assumptions of Theorem~\ref{thm:treelim} hold, and the local limit $\mathcal{T}(\boldsymbol{\mathcal{X}})$ is a (delayed) marked Galton-Watson tree, one can use the explicit characterization of the limiting opinion given in Proposition \ref{prop:explicit} to directly compute the variances of $\mathcal{M}^*$ and $\mathcal{R}^*$. The variance of $\mathcal{R}$ was shown in Proposition~\ref{prop:consensus} to be:
\begin{align}\label{eq:A}
    \var(\mathcal{R}^*)  & = \frac{d^2}{(c+d)^2}\var(\mathcal{Z}_\emptyset)   +  \frac{ \rho_2^*d^2}{(c+d)^2((c+d)^2-\rho_2)} \var(\mathcal{Z}_1).
\end{align}
The detailed derivation of the variance of $\mathcal{M}^*$ (including the case $P(\mathcal{N}_\emptyset = 0) + P(\mathcal{N}_1 = 0) > 0$) can be found in the appendix (see Remark~5), and yields
\begin{align}\label{eq:B}
    \var(\mathcal{M}^*)  &= \frac{d^2}{(c+d)^2}\var(\mathcal{Z}_\emptyset) + \frac{d^2\rho_2^*}{(c+d)^2((c+d)^2-\rho_2)}\var(\mathcal{Z}_1) \notag \\
    & \hspace{5mm} +  \frac{2d^2(c+d-1)}{(c+d)^2(2-(c+d))}E[\var(\mathcal{Z}_\emptyset|\boldsymbol{\mathcal{X}}_\emptyset)] \notag\\
    &\hspace{5mm} +d^2\rho_2^*\left(\frac{1}{1-\rho_2}E\left[ (c+d)^{-2(T+2)}p_{T+1} \right] - \frac{1}{(c+d)^2((c+d)^2-\rho_2)}\right)E[\var(\mathcal{Z}_1|\boldsymbol{\mathcal{X}}_1)] \notag.
\end{align}
where $p_{T+1} = \sum_{s=0}^\infty \binom{s+T+1}{s}^2 (1-c-d)^{2s} (c+d)^{2(T+2)}$, and $T$ is a geometric random variable (on $\{0,1,2,\dots\}$) with success probability $1-\rho_2$. Since $E[\var(\mathcal{Z}_\emptyset|\boldsymbol{\mathcal{X}}_\emptyset)]$ and $E[\var(\mathcal{Z}_1|\boldsymbol{\mathcal{X}}_1)]$ are always positive, it only remains to show that their corresponding coefficients are negative. The coefficient of $E[\var(\mathcal{Z}_\emptyset|\boldsymbol{\mathcal{X}}_\emptyset)]$ is clearly negative since $c+d < 1$. To see that the coefficient of $E[\var(\mathcal{Z}_1|\boldsymbol{\mathcal{X}}_1)]$ is also negative, note that since $p_{T+1}\in(0,1)$ we obtain that
\begin{align*}
    E[(c+d)^{-2(T+2)}p_{T+1}] & \leq  E[(c+d)^{-2(T+2)}] \\
    & = \frac{1}{(c+d)^4} \sum_{k = 0}^\infty \rho_2^{k}(1-\rho_2) (c+d)^{-2k}
    = \frac{1-\rho_2}{(c+d)^2((c+d)^2-\rho_2)}.
\end{align*}
We have thus shown that
\begin{equation*}
    \frac{1}{1-\rho_2}E\left[ (c+d)^{-2(T+2)}p_{T+1}\right] - \frac{1}{(c+d)^2((c+d)^2-\rho_2)}\leq 0. 
\end{equation*}
It follows now by comparing to the expression in \eqref{eq:A} that
$$\var(\mathcal{M}^*) \leq \frac{d^2}{(c+d)^2}\var(\mathcal{Z}_\emptyset) + \frac{d^2\rho_2^*}{(c+d)^2((c+d)^2-\rho_2)}\var(\mathcal{Z}_1) = \var(\mathcal{R}^*).$$
\end{proof}
\section*{Acknowledgments}
Tzu-Chi Lin has been partially supported by NSF DMS-1929298, awarded to the Statistical and Applied Mathematical Sciences Institute, and NSF CMMI-2243261.

\section*{Appendix}

In the appendix, we relax the assumptions in Proposition~\ref{prop:consensus} and compute the mean and variance of $\mathcal{R}^*$ for the general recursion~\eqref{eq:OpinionRec} (which may have $c+d < 1$ and $P(\mathcal{N}_\emptyset = 0) + P(\mathcal{N}_1 = 0) > 0$) on any directed random graph whose local weak limit $\mathcal{T}(\boldsymbol{\mathcal{X}})$ is a (delayed) marked Galton-Watson tree. The node marks $\{\boldsymbol{\mathcal{X}}_\mathbf{i}: \mathbf{i} \in \mathcal{U}\}$ in a (delayed) marked Galton-Watson tree are independent of each other, with $\{\boldsymbol{\mathcal{X}}_\mathbf{i}: \mathbf{i} \in \mathcal{U}, \mathbf{i} \neq \emptyset\}$ i.i.d. This structure allows us to directly compute the mean and variance of $\mathcal{R}^*$ using \eqref{eq:sol_memory} in Proposition~\ref{prop:explicit}:
$$
\mathcal{R}^*: =  \sum_{s=0}^{\infty} \sum_{l=0}^{s}  \sum_{|\mathbf{j}| = l} \Pi_{\mathbf{j}} a_{l,s} \mathcal{W}_{\mathbf{j}}^{(s)},
$$
where $a_{l,s} = \binom{s}{l}(1-c-d)^{s-l}$, and $\Pi_{\mathbf{j}}$ is recursively defined by $\Pi_{(\mathbf{i}, j)} = \Pi_{\mathbf{i}}\mathcal{C}_{(\mathbf{i},j)}$ with $\Pi_\emptyset = 1$.

\begin{theorem}\label{thm:LimitR_MeanVariance}
Suppose the assumptions of Theorem~\ref{thm:treelim} hold, with the local weak limit being a (delayed) marked Galton-Watson tree $\mathcal{T}(\boldsymbol{\mathcal{X}})$. If we let $\mathcal{R}^*$ denote the limiting opinion of the root node in $\mathcal{T}(\boldsymbol{\mathcal{X}})$ and $\mathcal{R}$ denote the limiting opinion of its neighbors, then the mean and variance are given by
\begin{align*}
E\left[ \mathcal{R}^*\right] &= \frac{1}{c+d}E[\mathcal{W}_\emptyset]  + \frac{\rho_1^*}{(c+d-\rho_1) (c+d)} E[\mathcal{W}_1], \\
E\left[ \mathcal{R} \right] &= \frac{E\left[\mathcal{W}_1\right] }{c+d-\rho_1}, 
\end{align*}
where $\rho_1^* = E\left[\sum_{i = 1}^{\mathcal{N}_\emptyset}\mathcal{C}_{i}\right]$ and $\rho_1 = E\left[\sum_{i = 1}^{\mathcal{N}_1}\mathcal{C}_{(1,i)}\right]$. Moreover, if we let 
$$
\rho_2^* = E\left[\sum_{i = 1}^{\mathcal{N}_\emptyset}\mathcal{C}_{i}^2\right],
\qquad
\rho_2 = E\left[\sum_{i = 1}^{\mathcal{N}_1}\mathcal{C}_{(1,i)}^2\right],
\qquad
\mathcal{Y}_\mathbf{i} = E[\mathcal{W}_\mathbf{i}^{(s)}|\boldsymbol{\mathcal{X}}_\mathbf{i}],
\qquad
\mathcal{V}_\mathbf{i} = \var(\mathcal{W}_\mathbf{i}^{(s)}|\boldsymbol{\mathcal{X}}_\mathbf{i}),
$$
then
\begin{align*}
\var\left(\mathcal{R}^*\right)  &=\var\left(\sum_{i = 1}^{\mathcal{N}_\emptyset} \mathcal{C}_{i}\right)  E[\mathcal{W}_1]^2 \frac{1}{(c+d)^2(c+d-\rho_1)^2} \\
& \hspace{5mm} +  \var(\mathcal{Y}_\emptyset)  \frac{1}{(c+d)^2}\\
& \hspace{5mm}+ 2\cov\left(\sum_{i = 1}^{\mathcal{N}_\emptyset} \mathcal{C}_{i}, \mathcal{Y}_\emptyset\right)E[\mathcal{W}_1] \frac{1}{(c+d)^2(c+d-\rho_1)} \\
& \hspace{5mm} + E[\mathcal{V}_\emptyset] \left( \frac{1}{1-(1-c-d)^2} \right)\\
& \hspace{5mm} + \rho_2^* \var\left(\sum_{i = 1}^{\mathcal{N}_1} \mathcal{C}_{(1,i)}\right)E[\mathcal{W}_1]^2\frac{1}{(c+d)^2(c+d- \rho_1)^2((c+d)^2 - \rho_2)}\\
&\hspace{5mm} + \rho_2^*E[\mathcal{V}_1] \frac{1}{1-\rho_2}E\left[ (c+d)^{-2(T+2)}p_{T+1}\right] \\
& \hspace{5mm}+ \rho_2^*\var(\mathcal{Y}_1) \frac{1}{(c+d)^2((c+d)^2-\rho_2)}\\
&\hspace{5mm} + 2 \rho_2^* \cov\left(\sum_{i = 1}^{\mathcal{N}_1} \mathcal{C}_{(1,i)}, \mathcal{Y}_1\right)   E[\mathcal{W}_1] \frac{1}{(c+d)^2(c+d-\rho_1)((c+d)^2-\rho_2)},  \\
\var\left(\mathcal{R}\right) &= \var\left(\sum_{i = 1}^{\mathcal{N}_1} \mathcal{C}_{(1,i)}\right)  E[\mathcal{W}_1]^2 \frac{1}{\left((c+d)- \rho_1\right)^2((c+d)^2 - \rho_2)} \\
&\hspace{5mm} +E[\mathcal{V}_1] \frac{1}{1-\rho_2}E\left[ (c+d)^{-2(T+1)}p_T\right]  + \var(\mathcal{Y}_1)  \frac{1}{(c+d)^2-\rho_2} \\
&\hspace{5mm} + 2 \cov\left(\sum_{i = 1}^{\mathcal{N}_1} \mathcal{C}_{(1,i)}, \mathcal{Y}_1\right)  E[\mathcal{W}_1] \frac{1}{(c+d-\rho_1)((c+d)^2-\rho_2)},
\end{align*}
where $p_T =  \sum_{s = 0}^{\infty} \binom{s+T}{s}^2(1-c-d)^{2s}(c+d)^{2(T+1)} \in (0,1)$ with $T$ being a Geometric random variable (on $\{0, 1, 2,\dots\}$) with success probability $1-\rho_2$.
\end{theorem}

\begin{remarks} \label{Rem:NoBotsMeanVar} \
\begin{enumerate}[leftmargin=*]
    \item In the case when $\boldsymbol{\mathcal{X}}_\emptyset \stackrel{\mathcal{D}}{=} \boldsymbol{\mathcal{X}}_1$, i.e., there is no size-bias, then $\mathcal{R}^* \stackrel{\mathcal{D}}{=} \mathcal{R}$ and the formulas for the mean and variance of $\mathcal{R}^*$ are simpler.
    \item If we choose $
W_i^{(k)} = d Z_i^{(k)}  + Q_i \left( c - \sum_{j=1}^{N_i} C(i,j) \right)$ and $P(\mathcal{N}_\emptyset>0) = P(\mathcal{N}_1>0) = 1$, the mean and variance of $\mathcal{R}^*$ reduce to:
\begin{align*}
    E\left[ \mathcal{R}^*\right] &= \frac{1}{c+d}E[\mathcal{Z}_\emptyset]  + \frac{\rho_1^*}{(c+d-\rho_1) (c+d)} E[\mathcal{Z}_1], \\
    \var(\mathcal{R}^*)  &= \frac{d^2}{(c+d)^2}\var(\mathcal{Z}_\emptyset) + \frac{d^2\rho_2^*}{(c+d)^2((c+d)^2-\rho_2)}\var(\mathcal{Z}_1) \notag \\
    & \hspace{5mm} +  \frac{2d^2(c+d-1)}{(c+d)^2(2-(c+d))}E[\var(\mathcal{Z}_\emptyset|\boldsymbol{\mathcal{X}}_\emptyset)] \notag\\
    &\hspace{5mm} +d^2\rho_2^*\left(\frac{1}{1-\rho_2}E\left[ (c+d)^{-2(T+2)}p_{T+1}\right] - \frac{1}{(c+d)^2((c+d)^2-\rho_2)}\right)E[\var(\mathcal{Z}_1|\boldsymbol{\mathcal{X}}_1)].
\end{align*}
\end{enumerate}

\end{remarks}

The proof of Theorem~\ref{thm:LimitR_MeanVariance} is separated into two steps. In the first step (Lemma~\ref{lem:Rk_MeanVariance} below), we compute the mean and variance of $\mathcal{R}_\emptyset^{(k)}$ for any fixed $k \geq 1$. Then, in the main proof of Theorem~\ref{thm:LimitR_MeanVariance}, we take the limit as $k \to \infty$.

\begin{lemma} \label{lem:Rk_MeanVariance}
Suppose $\mathcal{R}_\mathbf{i}^{(0)} \equiv 0$ for all $\mathbf{i} \in \mathcal{T}(\boldsymbol{\mathcal{X}})$. Let $\rho_1^*, \, \rho_1, \, \rho_2^*, \, \rho_2, \, \mathcal{Y}_\mathbf{i}$ and $\mathcal{V}_\mathbf{i}$ be defined as in Theorem~\ref{thm:LimitR_MeanVariance}. 
Then, for any $k \geq 0$,
\begin{align*}
E\left[ \mathcal{R}_\emptyset^{(k+1)} \right] &= E[\mathcal{W}_1] \frac{\rho_1^*}{\rho_1}\sum_{s=0}^{k} (1-c - d+\rho_1)^s + \left( E[\mathcal{W}_\emptyset] - E[\mathcal{W}_1] \frac{\rho_1^*}{\rho_1} \right) \sum_{s=0}^{k} (1-c-d)^s,\\
E\left[ \mathcal{R}_\mathbf{i}^{(k+1)} \right] &= E[\mathcal{W}_1]  \sum_{s=0}^{k} (1- c - d + \rho_1)^s, \qquad \mathbf{i} \neq \emptyset.
\end{align*}
Moreover, for $k \geq 0$:
\begin{align*}
\var\left(\mathcal{R}_\emptyset^{(k+1)} \right) &= E[\mathcal{W}_1]^2 \var\left(  \sum_{i=1}^{\mathcal{N}_\emptyset} \mathcal{C}_{i}\right) S_{k,0}^2 + U_{k,0}^2 \var(\mathcal{Y}_\emptyset) + 2 E[\mathcal{W}_1] S_{k,0} U_{k,0} \cov\left( \sum_{i=1}^{\mathcal{N}_\emptyset} \mathcal{C}_{i}, \mathcal{Y}_\emptyset\right) + E[\mathcal{V}_\emptyset] T_{k,0} \\
&\hspace{5mm} + \rho_2^*\rho_2^{k-1} \var(\mathcal{W}_1) +  \rho_2^* E[\mathcal{W}_1]^2 \var\left(\sum_{i = 1}^{\mathcal{N}_1} \mathcal{C}_{(1,i)}\right) \sum_{m=1}^{k-1} \rho_2^{m-1} S_{k-m,m}^2 \\
&\hspace{5mm} + \rho_2^* E[\mathcal{V}_1]  \sum_{m=1}^{k-1} \rho_2^{m-1} T_{k-m,m}  +  \rho_2^* \var(\mathcal{Y}_1) \sum_{m=1}^{k-1} \rho_2^{m-1} U_{k-m,m}^2 \\
&\hspace{5mm} + 2 \rho_2^* E[\mathcal{W}_1] \cov\left(\sum_{i = 1}^{\mathcal{N}_1} \mathcal{C}_{(1,i)}, \mathcal{Y}_1\right)  \sum_{m=1}^{k-1} \rho_2^{m-1} S_{k-m,m} U_{k-m,m},\\
\var\left(\mathcal{R}_\mathbf{i}^{(k+1)} \right) &=    \rho_2^k \var(\mathcal{W}_1) + E[\mathcal{W}_1]^2 \var\left(\sum_{i = 1}^{\mathcal{N}_1} \mathcal{C}_{(1,i)}\right)  \sum_{m=0}^{k-1} \rho_2^m S_{k-m,m}^2 + E[\mathcal{V}_1] \sum_{m=0}^{k-1} \rho_2^m T_{k-m,m} \\
&\hspace{5mm} + \var(\mathcal{Y}_1) \sum_{m=0}^{k-1} \rho_2^m U_{k-m,m}^2 + 2 E[\mathcal{W}_1] \cov\left(\sum_{i = 1}^{\mathcal{N}_1} \mathcal{C}_{(1,i)}, \mathcal{Y}_1\right) \sum_{m=0}^{k-1} \rho_2^m S_{k-m,m} U_{k-m,m},
\end{align*}
where
\begin{align*}
S_{k,m} &=  \sum_{s=1}^{k} \sum_{l=1}^s a_{l+m,s+m} \rho_1^{l-1}, \\
U_{k,m} &= \sum_{s=0}^k a_{m, s+m}, \\
T_{k,m} &=  \sum_{s=0}^k a_{m,s+m}^2.
\end{align*}
\end{lemma}

\begin{proof}
To start, recall that in a delayed Galton-Watson process the mark of the root node, $\boldsymbol{\mathcal{X}}_\emptyset$ is allowed to be different from that of all other nodes (a feature needed to model the size-bias incurred while exploring graphs). The sequence of marks $\{ \boldsymbol{\mathcal{X}}_\mathbf{i}: \mathbf{i} \in \mathcal{U}\}$ are assumed to be independent of each other. The main idea behind the proofs of the expressions for $E[\mathcal{R}_\emptyset^{(k)}]$ and $\var(\mathcal{R}_\emptyset^{(k)})$ is to derive recursive relations based on the tree's structure. The first observation we need to make is to use Proposition~\ref{prop:explicit} and the i.i.d.~nature of the noises given the vertex marks to obtain:   
$$
\mathcal{R}_\emptyset^{(k+1)} = \sum_{s=0}^{k} \sum_{l=0}^s \sum_{ |\mathbf{j}| = l } \Pi_{\mathbf{j}} a_{l,s} \mathcal{W}_\mathbf{j}^{(k-s)}  \eqlaw \sum_{s=0}^{k} \sum_{l=0}^s \sum_{ |\mathbf{j}| = l} \Pi_{\mathbf{j}} a_{l,s} \mathcal{W}_\mathbf{j}^{(s)}.
$$

To compute the mean, note that
\begin{align}
E\left[\mathcal{R}_\emptyset^{(k+1)} \right] &= \sum_{s=0}^{k} \sum_{l=1}^s a_{l,s} E\left[ \sum_{ |\mathbf{j}| = l} \Pi_{\mathbf{j}} \mathcal{W}_\mathbf{j}^{(s)} \right]  + \sum_{s=0}^{k} a_{0,s} E\left[ \mathcal{W}_\emptyset^{(s)} \right] \notag \\
&= E[ \mathcal{W}_1] \sum_{s=0}^{k} \sum_{l=1}^s a_{l,s} E\left[ \sum_{|\mathbf{j}| = l} \Pi_{\mathbf{j}}  \right]  + E\left[ \mathcal{W}_\emptyset \right] \sum_{s=0}^{k} a_{0,s}. \label{eq:MeanFromRoot}
\end{align}
We will now derive a recursion to compute the expectation of the sum of the weights, for which it is useful to remove the influence of the root node's mark first. To this end, define 
\begin{equation} \label{eq:MeanNoRoot}
b_0 : = 1, \quad \text{and} \quad b_l :=  E\left[ \sum_{|(1,\mathbf{j})| = l+1} \frac{\Pi_{(1,\mathbf{j})}}{\mathcal{C}_1} \right], \qquad l \geq 1,
\end{equation}
and recall that 
$$
\rho_1 =E\left[\sum_{m = 1}^{\mathcal{N}_1}\mathcal{C}_{(1,m)} \right] = cP(\mathcal{N}_1>0).
$$ 
Expanding the sum and then conditioning on $\boldsymbol{\mathcal{X}}_1$, $b_l$ can be rewritten as
\begin{align*}
b_l &= E\left[  \sum_{m=1}^{\mathcal{N}_1} \mathcal{C}_{(1,m)} \sum_{|(1, m, \mathbf{j})| = l+1 } \frac{ \Pi_{(1,m,\mathbf{j})} }{\Pi_{(1,m)}}  \right]  = E\left[  \sum_{m=1}^{\mathcal{N}_1} \mathcal{C}_{(1,m)} E\left[ \sum_{|(1, \mathbf{j}) | = l} \frac{ \Pi_{(1,\mathbf{j})} }{\mathcal{C}_1} \right] \right] = \rho_1 b_{l-1}. 
\end{align*}
Iterating the recursion for $b_l$ gives:
$$
b_l = \rho_1^{l}, \qquad l \geq 1.
$$
We now compute the expectation including the root's mark by noting that
$$E\left[ \sum_{|\mathbf{j}| = l} \Pi_{\mathbf{j}}  \right] = E\left[ \sum_{m=1}^{\mathcal{N}_\emptyset} \mathcal{C}_{m} \sum_{|(m,\mathbf{j})| = l} \frac{\Pi_{(m,j)}}{\mathcal{C}_m} \right] = E\left[ \sum_{m=1}^{\mathcal{N}_\emptyset} \mathcal{C}_{m} \rho_1^{(l-1)^+}   \right] = (\rho_1^*)^{l \wedge 1} \rho_1^{(l-1)^+}, \qquad l \geq 0, $$
where $\rho_1^* = E\left[\sum_{m = 1}^{\mathcal{N}_{\emptyset}} \mathcal{C}_m \right]$. 
Now use \eqref{eq:MeanFromRoot} to obtain that
\begin{align*}
E\left[ \mathcal{R}_\emptyset^{(k+1)} \right] &= E[ \mathcal{W}_1] \sum_{s=0}^{k} \sum_{l=1}^s a_{l,s} {\left(\rho_1^*\right)}^{l\wedge 1} \rho_1^{(l-1)^+}+ E\left[ \mathcal{W}_\emptyset \right] \sum_{s=0}^{k} a_{0,s} \\
&= E[\mathcal{W}_1] \frac{\rho_1^*}{\rho_1} \sum_{s=0}^{k} \sum_{l=1}^s \binom{s}{l} (1-c-d)^{s-l}  \rho_1^l + E\left[ \mathcal{W}_\emptyset \right] \sum_{s=0}^{k} (1-c-d)^s \\
&= E[\mathcal{W}_1] \frac{\rho_1^*}{\rho_1} \sum_{s=0}^{k} \left( (1-c- d+\rho_1)^s - (1-c-d)^s \right) + E\left[ \mathcal{W}_\emptyset \right] \sum_{s=0}^{k} (1-c-d)^s \\
&= E[\mathcal{W}_1] \frac{\rho_1^*}{\rho_1}\sum_{s=0}^{k} (1-c - d+\rho_1)^s + \left( E[\mathcal{W}_\emptyset] - E[\mathcal{W}_1] \frac{\rho_1^*}{\rho_1} \right) \sum_{s=0}^{k} (1-c-d)^s. 
\end{align*}
This completes the proof for the mean of the root. For any other node $\mathbf{i} \neq \emptyset$ the entire subtree rooted at $\mathbf{i}$ consists of nodes whose marks are i.i.d., and therefore, we can simply replace $\rho_1^*$ and $E[\mathcal{W}_\emptyset]$ above with $\rho_1$ and $E\left[ \mathcal{W}_1 \right]$, which yields:
$$E\left[ \mathcal{R}_\mathbf{i}^{(k+1)} \right] = E[\mathcal{W}_1] \sum_{s=0}^{k} (1-c - d+\rho_1)^s, \qquad \mathbf{i} \neq \emptyset.$$

To compute the variance of the root we start by writing:
$$ \mathcal{R}_\emptyset^{(k+1)} \stackrel{\mathcal{D}}{=}    \sum_{i=1}^{\mathcal{N}_\emptyset} \mathcal{C}_{i} \sum_{s=1}^{k} \sum_{l=1}^s \sum_{ |(i,\mathbf{j})| = l} \frac{\Pi_{(i,\mathbf{j})}}{\mathcal{C}_{i}} a_{l,s} \mathcal{W}_{(i,\mathbf{j})}^{(s)} +  \sum_{s=0}^{k}  a_{0,s} \mathcal{W}_\emptyset^{(s)},$$
which allows us to separate the effect of the root as we did for the mean. Next, condition on $\boldsymbol{\mathcal{X}}_\emptyset$ and use the law of total variance, we obtain
\begin{align*}
\var\left( \mathcal{R}_\emptyset^{(k+1)} \right)
&= \var\left( E\left[ \left.  \sum_{i=1}^{\mathcal{N}_\emptyset} \mathcal{C}_{i} \sum_{s=1}^{k} \sum_{l=1}^s \sum_{ |(i,\mathbf{j})| = l} \frac{\Pi_{(i,\mathbf{j})}}{\mathcal{C}_{i}} a_{l,s} \mathcal{W}_{(i,\mathbf{j})}^{(s)} +  \sum_{s=0}^{k}  a_{0,s} \mathcal{W}_\emptyset^{(s)}  \right| \boldsymbol{\mathcal{X}}_\emptyset \right]  \right) \\
&\hspace{5mm} + E\left[ \var\left( \left.  \sum_{i=1}^{\mathcal{N}_\emptyset} \mathcal{C}_{i}\sum_{s=1}^{k} \sum_{l=1}^s \sum_{|(i,\mathbf{j})| = l} \frac{\Pi_{(i,\mathbf{j})}}{\mathcal{C}_{i}} a_{l,s} \mathcal{W}_{(i,\mathbf{j})}^{(s)} +  \sum_{s=0}^{k}  a_{0,s} \mathcal{W}_\emptyset^{(s)} \right| \boldsymbol{\mathcal{X}}_\emptyset \right) \right] \\
&= \var\left(   \sum_{i=1}^{\mathcal{N}_\emptyset} \mathcal{C}_{i} \sum_{s=1}^{k} \sum_{l=1}^s a_{l,s} E\left[   \sum_{|(1,\mathbf{j})| = l} \frac{\Pi_{(1,\mathbf{j})}}{\mathcal{C}_1} \mathcal{W}_{(1,\mathbf{j})}^{(s)}  \right] +  \sum_{s=0}^{k}  a_{0,s} E\left[ \left. \mathcal{W}_\emptyset^{(s)} \right| \boldsymbol{\mathcal{X}}_\emptyset \right]  \right) \\
&\hspace{5mm} + E\left[  \sum_{i=1}^{\mathcal{N}_\emptyset} \mathcal{C}_{i}^2 \var\left(  \sum_{s=1}^{k} \sum_{l=1}^s \sum_{|(1,\mathbf{j})| = l} \frac{\Pi_{(1,\mathbf{j})}}{\mathcal{C}_1} a_{l,s} \mathcal{W}_{(1,\mathbf{j})}^{(s)}  \right) +  \var\left( \left. \sum_{s=0}^{k}  a_{0,s} \mathcal{W}_\emptyset^{(s)} \right| \boldsymbol{\mathcal{X}}_\emptyset \right) \right] \\
&= \var\left(   \sum_{i=1}^{\mathcal{N}_\emptyset} \mathcal{C}_{i} E[\mathcal{W}_1] \sum_{s=1}^{k} \sum_{l=1}^s a_{l,s} E\left[   \sum_{|(1,\mathbf{j})| = l} \frac{\Pi_{(1,\mathbf{j})}}{\mathcal{C}_1} \right] +  \mathcal{Y}_\emptyset \sum_{s=0}^{k}  a_{0,s}  \right) \\
&\hspace{5mm} + E\left[   \sum_{i=1}^{\mathcal{N}_\emptyset} \mathcal{C}_{i}^2 \right]  \var\left(  \sum_{s=1}^{k} \sum_{l=1}^s \sum_{ |(1,\mathbf{j})| = l} \frac{\Pi_{(1,\mathbf{j})}}{\mathcal{C}_1} a_{l,s} \mathcal{W}_{(1,\mathbf{j})}^{(s)}  \right) +  E\left[ \mathcal{V}_\emptyset  \right]\sum_{s=0}^{k}  a_{0,s}^2,   
\end{align*}

where $\mathcal{Y}_\emptyset = E\left[ \left. \mathcal{W}_\emptyset^{(s)} \right| \boldsymbol{\mathcal{X}}_\emptyset \right]$ and $\mathcal{V}_\emptyset = \var\left( \left.  \mathcal{W}_\emptyset^{(s)} \right| \boldsymbol{\mathcal{X}}_\emptyset \right)$. Recall from our derivation of the mean that
$$
b_{l-1} = E\left[   \sum_{|(1,\mathbf{j})| = l} \frac{\Pi_{(1,\mathbf{j})}}{\mathcal{C}_1} \right] =  \rho_1^{l-1}.
$$

Since the $\{ \mathcal{W}_\mathbf{i}^{(s)} : s \geq 1\}$ are i.i.d., then
$$
 \sum_{s=1}^{k} \sum_{l=1}^s \sum_{ |(1,\mathbf{j})| = l} \frac{\Pi_{(1,\mathbf{j})}}{\mathcal{C}_1} a_{l,s} \mathcal{W}_{(i,\mathbf{j})}^{(s)} \stackrel{\mathcal{D}}{=}   \sum_{s=0}^{k-1} \sum_{l=0}^{s} \sum_{|\mathbf{j}| = l} \Pi_{\mathbf{j} } a_{l+1,s+1} \mathcal{W}_{\mathbf{j}}^{(s)}, 
$$
where the tree on the right-hand side has $\boldsymbol{\mathcal{X}}_\emptyset \eqlaw \boldsymbol{\mathcal{X}}_1$. Next, define for $k \geq 0, m \geq 1$:
\begin{align*}
x_{k,m} &: =  \var\left(  \sum_{s=0}^{k} \sum_{l=0}^{s} \sum_{|\mathbf{j}| = l} \Pi_{\mathbf{j}} a_{l+m,s+m} \mathcal{W}_{\mathbf{j}}^{(s)}  \right), \quad \text{where} \quad \boldsymbol{\mathcal{X}}_\emptyset \eqlaw \boldsymbol{\mathcal{X}}_1,
\end{align*}
and note that we can rewrite the variance of $\mathcal{R}_\emptyset^{(k+1)}$ as:
\begin{align*}
\var\left( \mathcal{R}_\emptyset^{(k+1)} \right) &=  \var\left(  \sum_{i=1}^{\mathcal{N}_\emptyset} \mathcal{C}_{i} E[\mathcal{W}_1] \sum_{s=1}^{k} \sum_{l=1}^s a_{l,s} \rho_1^{l-1} +  \mathcal{Y}_\emptyset \sum_{s=0}^{k}  a_{0,s}  \right)  + E\left[ \mathcal{V}_\emptyset  \right]\sum_{s=0}^{k}  a_{0,s}^2  + \rho_2^*  x_{k-1,1} . 
\end{align*}
The same arguments used to derive the above also yield
\begin{align*}
x_{k-1,1} &=  \var\left(  \sum_{i=1}^{\mathcal{N}_1} \mathcal{C}_{(1,i)} E[\mathcal{W}_1] \sum_{s=1}^{k-1} \sum_{l=1}^s a_{l+1,s+1} \rho_1^{l-1}+  \mathcal{Y}_1 \sum_{s=0}^{k-1}  a_{1,s+1}  \right)  + \rho_2  x_{k-2,2} +  E\left[ \mathcal{V}_1  \right]\sum_{s=0}^{k-1}  a_{1,s+1}^2 \\
&=: B_{k-1,1} + \rho_2 x_{k-2,2},
\end{align*}
where for any $k \geq 0, m \geq 1$:
\begin{align*}
B_{k,m} &= \var\left( \sum_{i = 1}^{\mathcal{N}_1} \mathcal{C}_{(1,i)} E[\mathcal{W}_1] \sum_{s=1}^{k} \sum_{l=1}^s a_{l+m,s+m} \rho_1^{l-1} +  \mathcal{Y}_1 \sum_{s=0}^{k}  a_{m,s+m}  \right) + E[\mathcal{V}_1] \sum_{s=0}^{k} a_{m,s+m}^2 \\
&= S_{k,m}^2 E[\mathcal{W}_1]^2 \var\left(\sum_{i = 1}^{\mathcal{N}_1} \mathcal{C}_{(1,i)}\right) + U_{k,m}^2 \var (\mathcal{Y}_1) + 2 S_{k,m} U_{k,m} E[\mathcal{W}_1] \cov\left(\sum_{i = 1}^{\mathcal{N}_1} \mathcal{C}_{(1,i)}, \mathcal{Y}_1\right) + E[\mathcal{V}_1] T_{k,m}
\end{align*}
where
\begin{align*}
 S_{k,m} =  \sum_{s=1}^{k} \sum_{l=1}^s a_{l+m,s+m} \rho_1^{l-1}, \quad U_{k,m} = \sum_{s=0}^k a_{m, s+m}, \quad \text{and} \quad T_{k,m} =  \sum_{s=0}^k a_{m,s+m}^2.
\end{align*}

It follows from iterating the recursion that
\begin{align*}
x_{k-1,1} &= \rho_2 x_{k-2,2} + B_{k-1,1} = \rho_2^2 x_{k-3,3} + \rho_2 B_{k-2,2} + B_{k-1,1} = \rho_2^{k-1} x_{0,k} + \sum_{m=1}^{k-1} \rho_2^{m-1} B_{k-m, m} \\
&= \rho_2^{k-1} \var\left( a_{k,k} \mathcal{W}_1^{(0)} \right) + \var\left(\sum_{i = 1}^{\mathcal{N}_1} \mathcal{C}_{(1,i)}\right) E[\mathcal{W}_1]^2 \sum_{m=1}^{k-1} \rho_2^{m-1} S_{k-m,m}^2 + E[\mathcal{V}_1] \sum_{m=1}^{k-1} \rho_2^{m-1} T_{k-m,m}  \\
&\hspace{5mm} + \var(\mathcal{Y}_1) \sum_{m=1}^{k-1} \rho_2^{m-1} U_{k-m,m}^2  + 2 E[\mathcal{W}_1] \cov\left(\sum_{i = 1}^{\mathcal{N}_1} \mathcal{C}_{(1,i)}, \mathcal{Y}_1\right)   \sum_{m=1}^{k-1} \rho_2^{m-1} S_{k-m,m} U_{k-m,m} .
\end{align*}
Combining the above with our previous computations, we obtain:
\begin{align*}
\var\left(\mathcal{R}_\emptyset^{(k+1)} \right) &= E[\mathcal{W}_1]^2 \var\left(  \sum_{i=1}^{\mathcal{N}_\emptyset} \mathcal{C}_{i}\right) S_{k,0}^2 + U_{k,0}^2 \var(\mathcal{Y}_\emptyset) + 2 S_{k,0} U_{k,0} E[\mathcal{W}_1] \cov\left( \sum_{i=1}^{\mathcal{N}_\emptyset} \mathcal{C}_{i}, \mathcal{Y}_\emptyset\right) + E[\mathcal{V}_\emptyset] T_{k,0} \\
&\hspace{5mm} + \rho_2^*\rho_2^{k-1} \var(\mathcal{W}_1) +  \rho_2^* \var\left(\sum_{i = 1}^{\mathcal{N}_1} \mathcal{C}_{(1,i)}\right) E[\mathcal{W}_1]^2 \sum_{m=1}^{k-1} \rho_2^{m-1} S_{k-m,m}^2 \\
&\hspace{5mm} + \rho_2^* E[\mathcal{V}_1]  \sum_{m=1}^{k-1} \rho_2^{m-1} T_{k-m,m}  +  \rho_2^* \var(\mathcal{Y}_1) \sum_{m=1}^{k-1} \rho_2^{m-1} U_{k-m,m}^2 \\
&\hspace{5mm} + 2 \rho_2^* E[\mathcal{W}_1]  \cov\left(\sum_{i = 1}^{\mathcal{N}_1} \mathcal{C}_{(1,i)}, \mathcal{Y}_1\right)  \sum_{m=1}^{k-1} \rho_2^{m-1} S_{k-m,m} U_{k-m,m}. 
\end{align*}
The formula for $\var(\mathcal{R}_\mathbf{i}^{(k+1)})$ for $\mathbf{i} \neq \emptyset$ is obtained by assuming $\boldsymbol{\mathcal{X}}_\emptyset\stackrel{\mathcal{D}}{=}\boldsymbol{\mathcal{X}}_1$ and simplifying the expression. This completes the proof. 
\end{proof}

We are now ready to prove Theorem~\ref{thm:LimitR_MeanVariance} giving the characterization of the mean and variance of the limiting opinion distribution. These are obtained by taking limits in the expression from Lemma \ref{lem:Rk_MeanVariance}.

\begin{proof}[Proof of Theorem~\ref{thm:LimitR_MeanVariance}.]
We start by computing the mean of the root. Using the expression derived in Lemma~\ref{lem:Rk_MeanVariance}, and taking the limit as $k \to \infty$ we obtain that the mean of the root can be written as:
\begin{align*}
E\left[\mathcal{R}^*\right] &= \lim_{k \to \infty} E[R_{\emptyset}^{(k)}] \\
&= E[W_1] \frac{\rho_1^*}{\rho_1}\lim_{k \to \infty}\sum_{s=0}^{k-1} (1-c - d+\rho_1)^s + \left( E[W_\emptyset] - E[W_1] \frac{\rho_1^*}{\rho_1} \right) \lim_{k \to \infty}\sum_{s=0}^{k-1} (1-c-d)^s \\
&= E[W_1] \frac{\rho_1^*}{\rho_1} \cdot 
\frac{1}{c+d-\rho_1}+ \left( E[W_\emptyset] - E[W_1] \frac{\rho_1^*}{\rho_1} \right) \frac{1}{c+d} \\
&= \frac{1}{c+d} E[W_\emptyset] + E[W_1]  \frac{\rho_1^*}{(c+d-\rho_1) (c+d)},
\end{align*}
where $\rho_1^* = E\left[\sum_{i = 1}^{\mathcal{N}_{\emptyset}} \mathcal{C}_{i} \right]$ and $ \rho_1 = E\left[\sum_{i = 1}^{\mathcal{N}_{1}}\mathcal{C}_{(1,i)}\right]$. 
Similarly, for $\mathbf{i} \neq \emptyset$, 
\begin{align*}
E\left[\mathcal{R}\right] = \lim_{k \to \infty} E\left[\mathcal{R}_{\mathbf{i}}^{(k)}\right] & = E\left[\mathcal{W}_1\right] \lim_{k \to \infty}\sum_{s=0}^{k} (1- c - d + \rho_1)^s = \frac{E\left[\mathcal{W}_1\right] }{c+d-\rho_1}.
\end{align*}
This completes the proof for the means. 

To compute the variances of the limiting opinions, the main idea is to use the characterizations for $\var\left(\mathcal{R}_\emptyset^{(k+1)} \right)$ and $\var\left(\mathcal{R}_\mathbf{i}^{(k+1)} \right)$ in Lemma~\ref{lem:Rk_MeanVariance}, take the limit as $k\to\infty$, and then apply the monotone convergence theorem to interchange the limits and the infinite sums in order to obtain the final expressions. We start by computing the variance of $\mathcal{R}$ since it has a simpler expression. Note that the bounded convergence theorem followed by the expressions derived in Lemma~\ref{lem:Rk_MeanVariance}, gives
\begin{align*}
\var\left(\mathcal{R} \right) & = \lim_{k \to \infty} \var\left(\mathcal{R}_\mathbf{i}^{(k+1)} \right) \\
& = \var(\mathcal{W}_1) \lim_{k \to \infty} \rho_2^k + E[\mathcal{W}_1]^2 \var\left(\sum_{i = 1}^{\mathcal{N}_1} \mathcal{C}_{(1,i)}\right)  \lim_{k \to \infty} \sum_{m=0}^{k-1} \rho_2^m S_{k-m,m}^2 + E[\mathcal{V}_1] \lim_{k \to \infty} \sum_{m=0}^{k-1}  \rho_2^m T_{k-m,m} \\
&\hspace{5mm} + \var(\mathcal{Y}_1) \lim_{k \to \infty} \sum_{m=0}^{k-1} \rho_2^m U_{k-m,m}^2 + 2 E[\mathcal{W}_1] \cov\left(\sum_{i = 1}^{\mathcal{N}_1} \mathcal{C}_{(1,i)}, \mathcal{Y}_1\right) \lim_{k \to \infty} \sum_{m=0}^{k-1} \rho_2^m S_{k-m,m} U_{k-m,m},
\end{align*}
where 
$$
\rho_2^* = E\left[\sum_{i = 1}^{\mathcal{N}_\emptyset}\mathcal{C}_{i}^2\right],
\qquad
\rho_2 = E\left[\sum_{i = 1}^{\mathcal{N}_1}\mathcal{C}_{(1,i)}^2\right],
\qquad
\mathcal{Y}_\mathbf{i} = E[\mathcal{W}_\mathbf{i}^{(s)}|\boldsymbol{\mathcal{X}}_\mathbf{i}],
\qquad
\mathcal{V}_\mathbf{i} = \var(\mathcal{W}_\mathbf{i}^{(s)}|\boldsymbol{\mathcal{X}}_\mathbf{i}),
$$
\begin{align*}
 S_{k,m} =  \sum_{s=1}^{k} \sum_{l=1}^s a_{l+m,s+m} \rho_1^{l-1}, \quad U_{k,m} = \sum_{s=0}^k a_{m, s+m}, \quad \text{and} \quad T_{k,m} =  \sum_{s=0}^k a_{m,s+m}^2.
\end{align*}
Since the assumptions of Theorem~\ref{thm:treelim} include $d > 0$, then $\rho_2 < 1$, which gives $\rho_2^k \to 0$ as $k \to \infty$. To compute the limit of $\sum_{m=0}^{k-1} \rho_2^m S_{k-m,m}^2$ as $k \to \infty$, it is useful to rewrite the sum as the expectation of a geometric random variable on $\{0,1,2,\dots\}$ with success probability $1-\rho_2$, which we will denote $T$, as follows:
\begin{align*}
 \sum_{m = 0}^{k-1} \rho_2^m S_{k-m,m}^2& =  \frac{1}{1-\rho_2} \sum_{m = 0}^{\infty} \rho_2^m (1- \rho_2)  \left(\sum_{s = 1}^{k-m}\sum_{l = 1}^s a_{l+m, s+ m}\rho_1^{l-1} \right)^21(m < k) \\
& = \frac{1}{1-\rho_2} E\left[ \left(\sum_{s = 1}^{k-T}\sum_{l = 1}^s a_{l+T, s+ T}\rho_1^{l-1} \right)^21(T < k)\right].
\end{align*}
To compute the expectation, define 
$$
X_k : = \sum_{s = 1}^{k-T}\sum_{l = 1}^s a_{l+T, s+ T}\rho_1^{l-1}1(T< k).
$$ 
Note that $\{X_k: k \geq T+1\} $ is a nonnegative and monotone increasing sequence of random variables, so by the monotone convergence theorem we have that $\lim_{k \to \infty} E[X_k^2] = E\left[ \lim_{k \to \infty} X_k^2 \right]$. Moreover, 
\begin{align*}
	 \lim_{k \to \infty} X_k & = \sum_{s = 1}^\infty \sum_{l = 1}^s \binom{s+T}{s-l} (1-c-d)^{s-l}\rho_1^{l-1} \\
	 & = \sum_{l = 1}^\infty \sum_{s= l}^\infty \binom{s+T}{s-l} (1-c-d)^{s-l}\rho_1^{l-1} \\
	 & = \sum_{l= 1}^\infty \rho_1^{l-1}(c+d)^{-T-l-1}\sum_{m = 0}^\infty \binom{m+l+T}{m} (1-c-d)^{m}(c+d)^{T+l+1} \\
	 & = \sum_{l= 1}^\infty \rho_1^{l-1}(c+d)^{-T-l-1} \\
	 & = \frac{1}{(c+d)^{T+2}(1-\rho_1/(c+d))},
\end{align*}
where in the third equality we multiplied and divided by $(c+d)^{T+l-1}$, and set $m = s-l$. It follows that
\begin{align*}
\lim_{k \to \infty} \sum_{m = 0}^{k-1} \rho_2^m S_{k-m,m}^2 & =  \frac{1}{1-\rho_2} \lim_{k \to \infty} E[X_k^2] = \frac{1}{(1-\rho_2)(c+d)^4(1-\rho_1/(c+d))^2} E\left[(c+d)^{-2T}\right]. 
\end{align*}
Now use the observation that $E[(c+d)^{-2T}] = (1-\rho_2)(1-\rho_2/(c+d)^2)^{-1}$ to obtain that
\begin{align*}
    \lim_{k \to \infty} \sum_{m = 0}^{k-1} \rho_2^m S_{k-m,m}^2 & = \frac{1}{\left(c+d- \rho_1\right)^2((c+d)^2 - \rho_2)}.
\end{align*}
The same arguments also yield:
\begin{align*} \lim_{k \to \infty} \sum_{m = 0}^{k-1} \rho_2^m U_{k-m,m}^2 & =  \frac{1}{1-\rho_2}E\left[ \lim_{k \to \infty}  \left(\sum_{s = 0}^{k-T} a_{T,s+T}\right)^21(T<  k)\right]\\
& =  \frac{1}{1-\rho_2}E\left[ \left(\sum_{s = 0}^\infty \binom{s+T}{T}(1-c-d)^s\right)^2\right]\\
& =  \frac{1}{1-\rho_2}E\left[ \left(\sum_{s = 0}^\infty \binom{s+T}{s}(1-c-d)^s\right)^2\right]\\
& =  \frac{1}{1-\rho_2}E\left[ \left((c+d)^{-T-1}\sum_{s = 0}^\infty \binom{s+T}{s}(1-c-d)^s(c+d)^{T+1}\right)^2\right]\\
& =  \frac{1}{(1-\rho_2)(c+d)^2}E\left[ (c+d)^{-2T}\right] \\
& = \frac{1}{(c+d)^2-\rho_2},\\
\lim_{k \to \infty} \sum_{m = 0}^{k-1} \rho_2^m S_{k-m,m} U_{k-m,m}& =  \frac{1}{(c+d)^3(1-\rho_2)(1-\rho_1/(c+d))} E\left[ (c+d)^{-2T}\right]\\
 & =  \frac{1}{(c+d-\rho_1)((c+d)^2-\rho_2)},
\end{align*}
and
\begin{align*}
 \lim_{k \to \infty} \sum_{m = 0}^{k-1} \rho_2^m T_{k-m,m}
& =  \frac{1}{1-\rho_2}E\left[ \lim_{k \to \infty} \sum_{s = 0}^{k-T} a_{T,s+T}^21(T < k)\right]\\
& =  \frac{1}{1-\rho_2}E\left[ \sum_{s = 0}^{\infty} \binom{s+T}{T}^2(1-c-d)^{2s}\right]\\
& =  \frac{1}{1-\rho_2}E\left[ (c+d)^{-2(T+1)}\sum_{s = 0}^{\infty} \binom{s+T}{s}^2(1-c-d)^{2s}(c+d)^{2(T+1)}\right]\\
& =  \frac{1}{1-\rho_2}E\left[ (c+d)^{-2(T+1)}p_T\right],
\end{align*}
where $p_T  = \sum_{s = 0}^{\infty} \binom{s+T}{s}^2(1-c-d)^{2s}(c+d)^{2(T+1)} \in (0,1)$ with $T$ a Geometric random variable (on $\{0, 1, 2,\dots\}$) with success probability $1-\rho_2$. Combining all the computations above, we obtain: 
\begin{align*}
\var\left(\mathcal{R}\right) &= \var\left(\sum_{i = 1}^{\mathcal{N}_1} \mathcal{C}_{(1,i)}\right)  E[\mathcal{W}_1]^2 \frac{1}{\left(c+d- \rho_1\right)^2((c+d)^2 - \rho_2)}\\
&\hspace{5mm} +E[\mathcal{V}_1] \frac{1}{1-\rho_2}E\left[ (c+d)^{-2(T+1)}p_T\right]  + \var(\mathcal{Y}_1)  \frac{1}{(c+d)^2-\rho_2} \\
&\hspace{5mm} + 2 \cov\left(\sum_{i = 1}^{\mathcal{N}_1} \mathcal{C}_{(1,i)}, Y_1\right)  E[\mathcal{W}_1] \frac{1}{(c+d-\rho_1)((c+d)^2-\rho_2)}.
\end{align*}
To compute the variance of $\mathcal{R}^*$, use the expression in Lemma~\ref{lem:Rk_MeanVariance} and take the limit as $k \to \infty$ to obtain:
\begin{align*}
\var\left(\mathcal{R}^*\right) & = \lim_{k \to \infty} \var\left(\mathcal{R}_\emptyset^{(k+1)} \right) \\
& = E[\mathcal{W}_1]^2  \var\left(  \sum_{i=1}^{\mathcal{N}_\emptyset} \mathcal{C}_{i}\right) \lim_{k \to \infty}S_{k,0}^2 + \var(\mathcal{Y}_\emptyset)\lim_{k \to \infty}U_{k,0}^2  + 2 E[\mathcal{W}_1] \cov\left( \sum_{i=1}^{\mathcal{N}_\emptyset} \mathcal{C}_{i}, \mathcal{Y}_\emptyset\right) \lim_{k \to \infty}S_{k,0} U_{k,0} \\
&\hspace{5mm} + E[\mathcal{V}_\emptyset] \lim_{k \to \infty}T_{k,0} 
 + \rho_2^*\var(\mathcal{W}_1)\lim_{k \to \infty}\rho_2^{k-1}  +  \rho_2^* E[\mathcal{W}_1]^2 \var\left(\sum_{i = 1}^{\mathcal{N}_1} \mathcal{C}_{(1,i)}\right) \lim_{k \to \infty}\sum_{m=1}^{k-1} \rho_2^{m-1} S_{k-m,m}^2 \\
&\hspace{5mm} + \rho_2^* E[\mathcal{V}_1]  \lim_{k \to \infty}\sum_{m=1}^{k-1} \rho_2^{m-1} T_{k-m,m}  +  \rho_2^* \var(\mathcal{Y}_1) \lim_{k \to \infty}\sum_{m=1}^{k-1} \rho_2^{m-1} U_{k-m,m}^2 \\
&\hspace{5mm} + 2 \rho_2^* E[\mathcal{W}_1] \cov\left(\sum_{i = 1}^{\mathcal{N}_1} \mathcal{C}_{(1,i)}, \mathcal{Y}_1\right)  \lim_{k \to \infty}\sum_{m=1}^{k-1} \rho_2^{m-1} S_{k-m,m} U_{k-m,m}.
\end{align*}
Let $T' \eqlaw T + 1$, where $T$ is a Geometric random variable (on $\{0,1,2,\dots\})$ with success probability $1-\rho_2$. Then the previous computations can be rewritten as:
\begin{align*}
\lim_{k \to \infty} \sum_{m = 1}^{k-1} \rho_2^{m-1} S_{k-m,m}^2 
 & = \frac{1}{(c+d)^2((c+d)-\rho_1)^2((c+d)^2 - \rho_2)},\\
  \lim_{k \to \infty} \sum_{m = 1}^{k-1} \rho_2^{m-1} T_{k-m,m} & =  \frac{1}{1-\rho_2}E\left[ (c+d)^{-2(T'+1)}p_{T'}\right]\\
 &  =  \frac{1}{1-\rho_2}E\left[ (c+d)^{-2(T+2)}p_{T+1}\right],\\
 \lim_{k \to \infty} \sum_{m = 1}^{k-1} \rho_2^{m-1} U_{k-m,m}^2 &  =  \frac{1}{(1-\rho_2)(c+d)^2}E\left[ (c+d)^{-2T'}\right] \\
& = \frac{1}{(c+d)^2((c+d)^2-\rho_2)},\\
 \lim_{k \to \infty} \sum_{m = 1}^{k-1} \rho_2^{m-1} S_{k-m,m} U_{k-m,m}& =  \frac{1}{(c+d)^3(1-\rho_2)(1-\rho_1/(c+d))} E\left[ (c+d)^{-2T'}\right]\\
 & =  \frac{1}{(c+d)^2(c+d-\rho_1)((c+d)^2-\rho_2)}.
\end{align*}
In addition to the above results, we also have:
\begin{align*}
\lim_{k \to \infty} S_{k,0} 
& = \frac{1}{\rho_1}\sum_{s = 1}^{\infty}\sum_{l = 1}^s \binom{s}{l} (1-c-d)^{s-l} \rho_1^{l}\\
& =  \frac{1}{\rho_1}\sum_{s = 1}^{\infty} \left(   (1-c-d + \rho_1)^{s} - (1-c-d)^s \right)\\
& =  \frac{1}{\rho_1} \left( \frac{ 1-c-d+ \rho_1}{c+ d-\rho_1}  - \frac{1-c-d}{c+d} \right)\\
& =  \frac{1}{(c+d)(c+d-\rho_1)}, \\
\lim_{k \to \infty} U_{k,0} & = \sum_{s = 0}^{\infty}  (1-c-d)^s = \frac{1}{c+d}, \text{ and}\\
\lim_{k \to \infty} T_{k,0} & = \sum_{s = 0}^{\infty}  (1-c-d)^{2s} = \frac{1}{1-(1-c-d)^2}.
\end{align*}
Combining all the limits computed above, we obtain that
\begin{align*}
\var\left(\mathcal{R}^*\right)  
&=\var\left(\sum_{i = 1}^{\mathcal{N}_\emptyset} \mathcal{C}_{i}\right)  E[\mathcal{W}_1]^2 \frac{1}{(c+d)^2(c+d-\rho_1)^2} \\
& \hspace{5mm} +  \var(\mathcal{Y}_\emptyset)  \frac{1}{(c+d)^2}\\
& \hspace{5mm}+ 2\cov\left(\sum_{i = 1}^{\mathcal{N}_\emptyset} \mathcal{C}_{i}, \mathcal{Y}_\emptyset\right)E[\mathcal{W}_1] \frac{1}{(c+d)^2(c+d-\rho_1)} \\
& \hspace{5mm} + E[\mathcal{V}_\emptyset] \left( \frac{1}{1-(1-c-d)^2} \right)\\
& \hspace{5mm} + \rho_2^* \var\left(\sum_{i = 1}^{\mathcal{N}_1} \mathcal{C}_{(1,i)}\right)E[\mathcal{W}_1]^2\frac{1}{(c+d)^2((c+d)-\rho_1)^2((c+d)^2 - \rho_2)}\\
&\hspace{5mm} + \rho_2^*E[\mathcal{V}_1] \frac{1}{1-\rho_2}E\left[ (c+d)^{-2(T+2)}p_{T+1}\right] \\
& \hspace{5mm}+ \rho_2^*\var(\mathcal{Y}_1) \frac{1}{(c+d)^2((c+d)^2-\rho_2)}\\
&\hspace{5mm} + 2 \rho_2^* \cov\left(\sum_{i = 1}^{\mathcal{N}_1} \mathcal{C}_{(1,i)}, \mathcal{Y}_1\right)   E[\mathcal{W}_1] \frac{1}{(c+d)^2(c+d-\rho_1)((c+d)^2-\rho_2)}.  \\
\end{align*}
This completes the proof. 
\end{proof}

\bibliographystyle{plain} 
\bibliography{DriftRecBib}

\end{document}